\documentclass[11pt,letterpaper]{amsart}
\usepackage{graphicx} 
\usepackage[utf8]{inputenc}
\usepackage{amsmath}
\usepackage{amsfonts}
\usepackage[x11names, svgnames, rgb]{xcolor}
\usepackage{multicol}
\usepackage{url}
\usepackage{amsthm}
\usepackage{hyperref}
\usepackage[margin=2cm]{geometry}
 \usepackage{amssymb}
 \usepackage{yhmath}
\input xy
\xyoption{all}
\usepackage{comment}
\usepackage{tikz}
\usetikzlibrary{snakes,arrows,shapes}

\newenvironment{blue}{\relax\color{blue}}{\hspace*{.5ex}\relax}
\newcommand{\bcp}{\begin{blue} }
\newcommand{\ecp}{\end{blue}}

\newtheorem{theorem}{Theorem}[section]
\newtheorem{proposition}[theorem]{Proposition}
\newtheorem{definition}[theorem]{Definition}
\newtheorem{exmp}{Example}[section]
\newtheorem{lemma}[theorem]{Lemma}
\newtheorem{coro}[theorem]{Corollary}
\newtheorem{conj}[theorem]{Conjecture}

\newcommand{\TT}{\mathcal{T}}

\newcommand{\DD}{\mathcal{D}}

\newcommand{\PP}{\mathcal{P}}
\newcommand{\Aut}{\mathrm{Aut}}
\newcommand{\Cor}{\mathrm{Cor}}

\newcommand{\Sym}{\mathrm{Sym}} 
 
\newcommand{\ZZ}{\mathbb{Z}}
\newcommand{\FF}{\mathbb{F}}

\newcommand{\UU}{\mathcal{U}}

\title{The Geometry of Pattern Groups}

\begin{document}

\author[João Dias]{João Dias}
\address[João Dias]{Universidade de Évora, Escola de Ciências e Tecnologia, Departamento de Matemática, Centro de Investigacão em Matemática e Aplicacões
(CIMA), Orcid number 0000-0001-9400-7323
}
\email{joao.miguel.dias@uevora.pt}

\author[Claudio Alexandre Piedade]{Claudio Alexandre Piedade}
\address[Claudio Alexandre Piedade]{Université Libre de Bruxelles, Brussels, Belgium, Orcid number 0000-0002-0746-5893
}
\email{claudio.piedade@ulb.be}

{}

\date{\today}

\begin{abstract}
Coset incidence geometries are an important tool which connect group theory and geometry. While the representations of the unitriangular group and its pattern subgroups have been studied extensively, the underlying geometric structures of these groups has remained largely unexplored.

In this article, we construct the natural coset geometry associated with each pattern group over a finite field $\FF_q$, and prove geometrical properties of theses structures. We show how combinatorial features of their defining closed sets directly correspond to properties of their parabolic subgroups, in particular the intersection of parabolics and normality of the Borel subgroup. Finally, we characterize the automorphisms of these geometries and examine how pattern geometries behave under change of field within the same characteristic $p$.
\end{abstract}

\maketitle

\section{Introduction}

Coset incidence systems are incidence systems that are constructed from a group $G$ together with a system $(G_i)_{i\in I}$ of subgroups of $G$ (see Section \ref{sec:back:subsec:incidence} for precise definitions).
These abstract structures were introduced by Jacques Tits in the 60s~\cite{Tits1957,Tits1963geometries} and are the basis for a diverse set of combinatorial and geometric objects, from abstract polytopes, vector and projective spaces, and even Tits' buildings~\cite{Handbook,Tits1974,ARP}. In particular, the theory of buildings is a crucial area which 
illustrates how combinatorial structures can emerge from algebraic objects~\cite{Handbook,KenBrownBook2}. Additionally, abstract regular polytopes (as introduced in~\cite{ARP}) are a geometrical object, built from coset incidence systems, which are in one-to-one correspondence with a particular class of groups known as string $C$-groups.

The versatility of coset incidence systems lies on the freedom of taking any group and any set of its subgroups, and obtaining a combinatorial structure. Recently, there has been active research on both characterising coset geometries for different groups~\cite{Tranchida2025,Remy2026, Fernandes2026, Leemans2025},
and the
development of techniques to study and construct new coset geometries~\cite{Zhang2026,Piedade2025, Piedade2026,wreathproductPiedadeTranchida}.
Because groups and their geometric structures are closely related, studying one reveals key information about the other, in particular whenever the group has a lot of combinatorial properties.

In recent years, the character theory of unitriangular groups over finite fields has received considerable attention, in particular the study of their complex irreducible characters. This interest can be traced back to a conjecture of G. Higman \cite{higman1960enumerating}, concerning the character theory of unitriangular groups. A significant development in this direction was made by C. André, who introduced a family of characters known as basic characters in \cite{andre1995basic}. André's construction makes extensive use of the combinatorial structure underlying the unitriangular groups. Subsequently, Diaconis and Isaacs \cite{diaconis2008supercharacters} developed the general theory of supercharacters, providing a broader framework in which André's basic characters arise naturally.

Another important class of subgroups of the unitriangular groups is given by the pattern subgroups. These subgroups were introduced by A. J. Weir \cite{weir1955sylow} under the name partition subgroups, and were later referred to as pattern subgroups by I. M. Isaacs \cite{isaacs2007counting}. Pattern subgroups are defined combinatorially from certain closed sets of positions in the upper-triangular part of a matrix. More precisely, a closed set satisfies a suitable closure condition with respect to the natural partial order on these positions, and each such closed set determines a subgroup of the unitriangular group. This correspondence between combinatorial data and subgroups provides a useful framework for studying their structure and representation theory.
Due to their close relationship with the ambient unitriangular groups and their easy combinatorial description, pattern subgroups have become an important object of study.

The work presented here started from a question of Persi Diaconis, at the Theoretical and Computational Algebra Conference 2025 (Évora, Portugal), posed to one of the presenters, about the maximum size of independent generating set of a matrix group. A generating set of a group being independent is directly linked with properties such as residual-connectedness of the natural coset incidence system one can build for a group using that generating set. This question generated some discussions between the two authors of this paper, one with more expertise in matrix groups such as the unitriangular group and its subgroups, the other on coset geometries. Our research started with studying only the properties of the coset geometry of the unitriangular group, which later was generalised to the study of all pattern groups.

In this paper, we prove that the natural coset incidence system built using a pattern group is a residually connected geometry on which the group acts flag-transitively. Additionally, we prove under which conditions its parabolic subgroups satisfy the intersection property (a stronger version independence of generators, see Section~\ref{sec:back:IP}). 
 Indeed, one can relate the intersection property and the normality of the Borel subgroup with combinatorial properties of the partial ordered set used to define the pattern group.

These results extend naturally to the coset geometries of the unitriangular groups, where we describe the Buekenhout diagram of these structures. Interestingly enough, we prove the existence of an infinite family of polytopes of type $\{4,\ldots,4\}$ not described so far in the literature.
Finally, we investigate the automorphisms of the unitriangular geometry and the relations between pattern groups over a field of the same characteristic $p$.

\subsection{Acknowledgments}
We would like to thank Marston Conder and Dong-Dong Hou for their input regarding the description in literature of the polytope of type $\{4,\ldots,4\}$ described in this paper.
Additionally, we would like to thank Dimitri Leemans and Philippe Tranchida for useful discussions regarding the automorphism group of coset geometries.

The first author acknowledges the support of FCT - Funda\c{c}\~ao para a Ci\^{e}ncia e Tecnologia under the project 10.54499/UIDB/04674/2020, and the research center CIMA -- Centro de Investigação em Matemática e Aplicações.
The second author is funded by an Action de Recherche Concert\'{e}e -- ARC -- from the Communaut\'{e}e Fran\c{c}aise Wallonie-Bruxelles and the second author is a Postdoctoral Researcher of the Fonds de la Recherche Scientifique -- FNRS.

\section{Background and Preliminary Results}
\subsection{Unitriangular Group and Pattern Groups}\label{sec:UTG_PG}
Let us denote by $UT(n,\FF_q)$ the group of upper unitriangular matrices of size $n$ over the field $\mathbb{F}_q$, where $q$ is a power of prime.
The unitriangular group is a $p$-group, that is, the order of $UT(n,\FF_q)$ is a power of a prime. In fact, this group has order $q^{\frac{n^2-n}{2}}$. 

Certain elements of the unitriangular group can be described in the following way:  given a $t\in \FF_q$ and $(i,j)\in\Delta_n=\{(a,b)\in\mathbb{N}^2\mid 1\leq a<b\leq n\}$, we define the element $e_{i,j}(t)$ as the unitriangular matrix such that:
$$(e_{i,j}(t))_{k,l}=\begin{cases}
    1,\quad k=l\\
    t,\quad k=i,\;l=j\\
    0,\quad \text{otherwise}
\end{cases}$$
Note that, for any $(i,j)\in\Delta_n$, $e_{i,j}(0) = \mathbf{1}$, the identity matrix of $UT(n,\FF_q)$.
Specifically, the matrices $e_{i,j}(t)$ satisfy certain relations, which are described below in Proposition~\ref{prop:gens_props}

\begin{proposition}\label{prop:gens_props}
    Let $(i,j)$ and $(k,l)\in\Delta_n$, and $t,t'\in \FF_q\setminus\{0\}$ then:
    \begin{enumerate}
        \item $e_{i,j}(t)e_{k,l}(t')=e_{k,l}(t')e_{i,j}(t)$ if and only if $i\neq l$ and $j\neq k$;
        \item $e_{i,j}(t)e_{i,j}(t')=e_{i,j}(t')e_{i,j}(t)=e_{i,j}(t+t')$;

        \item let $(i,j),(j,k),(i,k)\in \Delta_n$, then $e_{i,k}(tt')=e_{i,j}(t)e_{j,k}(t')e_{i,j}(-t)e_{j,k}(-t')=[e_{i,j}(t),e_{j,k}(t')]$;
        \item $e_{i,j}^{-1}(t) = e_{i,j}(-t)$
        \item $o(e_{i,j}(t)) = p$;
    \end{enumerate}
\end{proposition}
\begin{proof}
    Items $(1)$, $(2)$ and $(3)$ is common knowledge regarding upper unitriangular groups (see~\cite{biss2001presentation, Mahalanobis2013}) which can be easily derived. Indeed these relations are the Steinberg relations of unipotent groups.
    The relations $(4)$ can be derived from $(2)$ be taking $t'=-t$.

    For the remaining proof, let $f_{i,j}(t)$ be the matrix with zero everywhere, except the entry $(i,j)$, which has entry $t\in \FF_q\setminus\{0\}$. In other words, $f_{i,j}(t) = e_{i,j}(t)-\mathbf{1}$.
    Since $e_{i,j}(t)=\mathbf{1}+f_{i,j}(t)$, one easily can compute that $(e_{i,j}(t))^n=\mathbf{1}+nf_{i,j}(t)$, which implies that the order of $e_{i,j}$ must be $p$.
\end{proof}
Let $T = \{t_1,\dots,t_m\}$ be a basis for $\FF_q$ seen as a $\FF_p$-vector space, with $q=p^m$ for a prime $p$ and $m>1$.
The unitriangular group has the following set of classical generators 
$$
S_n(q)=\{e_{i,i+1}(t)\mid i\in \{1,\dots, n-1\}, t\in T\}.
$$
The set $S_n(q)$ is known to be a minimal generating set for $UT(n,\FF_q)$ (see Proposition~\ref{prop:Pattern_mingenset}).

Due to Proposition~\ref{prop:gens_props}(1), note that whenever $|i-j|>1$, we have that $e_{i,i+1}(t)$ and $e_{j,j+1}(t)$ will commute. The following proposition provides the order of the product of consecutive generators.

\begin{proposition}\label{prop:orderproductgens}
    Let $1 \leq i \leq n-1$ and $\alpha,\beta\in \FF_q\setminus\{0\}$,
    then
    \begin{enumerate}
        \item If $p=2$, then $o(e_{i,i+1}(\alpha)e_{i+1,i+2}(\beta))=o(e_{i+1,i+2}(\beta)e_{i,i+1}(\alpha))=4$ 
        \item If $p\neq2$, then $o(e_i(\alpha)e_{i+1}(\beta))=o(e_{i+1}(\beta)e_i(\alpha))=p$.
    \end{enumerate}
\end{proposition}
\begin{proof}
    In this proof, we will use the fact that $e_{i,j}(\alpha) = \mathbf{1} + f_{i,j}(\alpha)$. Note that $f_{i,j}(\alpha) f_{k,l}(\beta) = \delta_{j,k}f_{i,l}(\alpha\beta)$, where $\delta_{j,k}$ is the Kronecker delta.
    Let us consider first the case of $p=2$. Then, we have
    \begin{align*}
        e_{i,i+1}(\alpha)e_{i+1,i+2}(\beta)=(\mathbf{1}+f_{i,i+1}(\alpha))(\mathbf{1}+f_{i+1,i+2}(\beta))=\mathbf{1}+f_{i,i+1}(\alpha)+f_{i+1,i+2}(\beta)+f_{i,i+2}(\alpha\beta)
    \end{align*}
    By squaring we obtain
    \begin{align*}
        (e_{i,i+1}(\alpha)e_{i+1,i+2}(\beta))^2&=(\mathbf{1}+f_{i,i+1}(\alpha)+f_{i+1,i+2}(\beta)+f_{i,i+2}(\alpha\beta))^2\\&=\mathbf{1}+2(f_{i,i+1}(\alpha)+f_{i+1,i+2}(\beta)+f_{i,i+2}(\alpha\beta)+f_{i,i+2}(\alpha\beta)\\
        &=\mathbf{1}+f_{i,i+2}(\alpha\beta)
    \end{align*}
    And so we have
    \begin{align*}
    (e_{i,i+1}(\alpha)e_{i+1,i+2}(\beta))^4&= (\mathbf{1}+f_{i,i+2}(\alpha\beta))^2=\mathbf{1}
    \end{align*}
    And so in the case of $p=2$, the order of consecutive matrices $e_{i,i+1}(\alpha)$ and $e_{i+1,i+2}(\beta)$ is $4$.

    Consider now $p\neq 2$. Using similar arguments as before, we can obtain that
    \begin{align*}
        (e_{i,i+1}(\alpha)e_{i+1,i+2}(\beta))^n&=\mathbf{1}+n(f_{i,i+1}(\alpha)+f_{i+1,i+2}(\beta))+\frac{n(n+1)}{2}f_{i+1,i+2}(\alpha\beta)
    \end{align*}
    Then $(e_{i,i+1}(\alpha)e_{i+1,i+2}(\beta))^p=\mathbf{1}$, and so the order of consecutive matrices $e_{i,i+1}(\alpha)$ and $e_{i+1,i+2}(\beta)$ must be $p$. 
\end{proof}

One of the most important subgroups of the unitriangular group are called \emph{Pattern Groups}. Pattern groups sit in the intersection between representation theory and combinatorics~\cite{Dipper2017,Diaconis2009}. They play a very important role in the study of the character theory of the unitriangular group. To define these groups, first we need the notion of a (transitive) closed set. 
\begin{definition}\label{def:closure}
    Let $n\geq 1$ and $I\subseteq \Delta_n$. We say that $I$ is a \emph{transitive closed set}, or just \emph{closed set}, if $(i,j)\text{ and }(j,k)\in I$ implies $(i,k)\in I$.
\end{definition}

 For any set $I\subseteq  \Delta_n$, we shall denote by $\widehat{I}$ the transitive closure of $I$, so that $\widehat{I}$ is a closed set.
 Hence, from now on, any closed set will be always denoted with a ``hat", $\ \widehat{ }\ $, above it.

\begin{definition}[Pattern group]
    Let $n>1$ and $\widehat{I}\subseteq \Delta_n$ be a closed set. Then, the \emph{pattern group} associated to $\widehat{I}$ is the group:
    $$UT_{\widehat{I}}(n,\FF_q)=\{g\in UT(n,\FF_q)\mid g_{i,j}=0\text{ if }(i,j)\notin \widehat{I}\}$$
\end{definition}

The order of $UT_{\widehat{I}}(n,\FF_q)$ is $q^{|\widehat{I}|}$. For more details about Pattern groups, see~\cite{Diaconis2009}.
One unique property of pattern groups is that we can uniquely write a word to define each element. 
Indeed, as noted in~\cite{Dipper2017}, by fixing an arbitrary linear ordering $O_{\widehat{I}}$ and taking products following this ordering, one can uniquely write all elements of $UT_{\widehat{I}}(n,\FF_q)$. In other words, for any $g\in UT_{\widehat{I}}(n,\FF_q)$ and a linear ordering $O_{\widehat{I}}$,
$$ g = \prod_{(i,j)\in O_{\widehat{I}}}e_{i,j}(\alpha_{(i,j)})
\textnormal{, for some } \alpha_{(i,j)}\in \FF_q.$$
Additionally, due to the properties of Proposition~\ref{prop:gens_props}, this means that, for $g\in UT_{\widehat{I}}(n,\FF_q)$, the entry $g_{i,j} = 0$, for all $(i,j)\in \Delta_n\setminus  \widehat{I}$.

Given a closed set $\widehat{I}\subseteq\Delta_n$, one can define a partial ordered set $\PP_{\widehat{I}}:=(X,\prec)$, with ground set $X=\{1,\ldots,n\}$ and partial order given by $ \widehat{I}$.
For a finite partial ordered set $\PP_{\widehat{I}}$, one can take a particular subset of its partial order $\widehat{I}$, called the \emph{covering relations}, the set of relations $(i,j)$ such that there exists no $k\in\{1,\ldots,n\}$ such that $i\prec k\prec j$. For each closet set $\widehat{I}$, we will denote from now on the set of its covering relations as $I$.
With the set $I$, we can define a chain in $\PP_{\widehat{I}}$.

\begin{definition}
    Let $\widehat{I}$ be a closed subset of $\Delta_n$, with covering relations $I$.
    An ordered sequence $c=(c_1,c_2,\ldots,c_l)$, with each element $c_k\in\{1,\ldots,n\}$, $k\in\{1,\ldots,l\}$, is called a \emph{chain} in $\PP_{\widehat{I}}$ if, for $k\in\{1,\ldots,l-1\}$, each consecutive pair $(c_k,c_{k+1})$ in $c$ is in the set of covering relations $I$.
    Moreover, we say a chain is a \emph{$(i,j)$-chain} is $c_1=i$ and $c_l=j$.
    Finally, we define $C_{(i,j)}$ to be the set of all $(i,j)$-chain in $\widehat{I}$.
\end{definition} 
By definition of $\widehat{I}$ as a subset of $\Delta_n$, note that any $(i,j)$-chain is only defined if $i \prec j$ in $\PP_{\widehat{I}}$.

\begin{exmp}
    For $n=5$, consider 
    $\widehat{I}=\{(1,2),(1,3),(1,4),(1,5),(2,3),(2,4),(2,5),(3,5),(4,5)\}$. We can easily compute that the covering relations are $I=\{(1,2),(2,3),(2,4),(3,5),(4,5)\}$.
    We have that $(1,2,3,5)$ is a $(1,5)$-chain, while $(1,2,3,4,5)$ is not a chain, since $(3,4)$ is not in the set of covering relations.
    Finally, note that $C_{(1,5)}=\{(1,2,3,5),(1,2,4,5)\}$, i.e. these are the two possible distinct $(1,5)$-chains in $\widehat{I}$.
\end{exmp}
It is important to notice that that an element $(i,j)$ belongs to $\widehat{I}$ if an only if there exists an $(i,j)$-chain in $\PP_{\widehat{I}}$. Due to the property in Proposition~\ref{prop:gens_props}(3), the existence of an $(i,j)$-chain implies that one can write the matrix $e_{i,j}(t)$ from the covering relations $I$.
Taking this into account, it is direct that a
minimal generating set of a pattern group takes only entries in the covering relations.

\begin{proposition}\cite[Example 2]{Diaconis2009}
\label{prop:Pattern_mingenset}
    Let $\widehat{I}$ be a closed subset and consider $UT_{\widehat{I}}(n,\FF_q)$ the pattern group associated with $\widehat{I}$, and let $I$ be the set of covering relations. 
    Then $\{e_{i,j}(t)\mid (i,j)\in I, t\in T\}$ is a minimal generating set for $UT_{\widehat{I}}(n,\FF_q)$. 
\end{proposition}
Note that this means we can describe any element (possibly not uniquely) of $UT_{\widehat{I}}(n,\FF_q)$ as a product of matrices $e_{i,j}(t)$, with $(i,j)\in I$, the set of covering relations.

\begin{proposition}\label{prop:Cosets_representative}
    Let $\widehat{K}\subseteq \widehat{I}\subseteq \Delta_n$ be closed sets. For each $g\in UT_{\widehat{I}}(n,q)$, there exists a unique $h\in g UT_{\widehat{K}}(n,q)$ such that $h_{i,j}=0$ for all $(i,j)\in \widehat{K}$.
\end{proposition}
\begin{proof}
    Let $g\in UT_{\widehat{I}}(n,q)$. Any element $h\in gUT_{\widehat{K}}(n,q)$ can be written in the form $h=gu$ where $u\in UT_{\widehat{K}}(n,q)$, since $UT_{\widehat{K}}(n,q) \leq UT_{\widehat{I}}(n,q)$. 
    We want to show that there exists an unique $u\in UT_{\widehat{K}}(n,q)$ such that $h_{i,j}=0$ for all $(i,j)\in \widehat{K}$.

    Given any $(i,j)\in \Delta_n$, and since we are working with unitriangular matrices, we have the following:

    $$
    h_{i,j}=\sum_{i\leq t\leq j} g_{i,t}u_{t,j}=u_{i,j}+g_{i,j} +\sum_{i<t<j} g_{i,t}u_{t,j}
    $$
    Then there exists a unique $u\in UT_{\widehat{K}}(n,q)$ if and only if the following system of equations on $u_{i,j}$ has an unique solution.
    $$
    \begin{cases}
        u_{i,j}+g_{i,j} +\sum_{i<t<j} g_{i,t}u_{t,j}=0\quad &(i,j)\in \widehat{K}\\
        u_{i,j}=0 \quad &(i,j)\notin \widehat{K}
    \end{cases}
    $$
    
    We shall solve this system by induction on $j-i$.
    First suppose that $j-i=1$. Then, there exists no $t$ such that $i<t<j$, and so $\sum_{i<t<j}g_{i,t}u_{t,j}$ is the empty sum. Hence, we will obtain
        $$
    \begin{cases}
        u_{i,j}+g_{i,j}=0\quad &(i,j)\in \widehat{K}\\
        u_{i,j}=0 \quad &(i,j)\notin \widehat{K}
    \end{cases}
    $$
    In this case, $j-i=1$, there exists a unique solution.  

    Let $l\in \mathbb{N}$ such that $l>1$. Suppose, by the induction hypothesis, that for all $(r,s)\in\Delta_n$ with $s-r< l$ the following equations have unique solution.
    $$
    \begin{cases}
        u_{r,s}+g_{r,s} +\sum_{r<t<s} g_{r,t}u_{t,s}=0\quad &(r,s)\in \widehat{K}\\
        u_{r,s}=0 \quad &(r,s)\notin \widehat{K}.
    \end{cases}
    $$
    Consider now all $(i,j)\in \Delta$ such that $j-i=l$. Take the following equations for these $(i,j)$
    $$
    \begin{cases}
        u_{i,j}+g_{i,j} +\sum_{i<t<j} g_{i,t}u_{t,j}=0\quad &(i,j)\in \widehat{K}\\
        u_{i,j}=0 \quad &(i,j)\notin \widehat{K}
    \end{cases}
    $$
    and rearrange as follows
    $$
    \begin{cases}
        u_{i,j}=-g_{i,j} -\sum_{i<t<j} g_{i,t}u_{t,j}=0\quad &(i,j)\in \widehat{K}\\
        u_{i,j}=0 \quad &(i,j)\notin \widehat{K}.
    \end{cases}
    $$
    For all $t$ such that $i<t<j$, we have $j-t<l$, which by the induction hypothesis implies $u_{t,j}$ are uniquely determined. Since $u_{i,j}$ is a linear combination of such $u_{t,j}$ then $u_{i,j}$ is uniquely determined. And so, there exists an unique $u\in UT_{\widehat{K}}(n,q)$ such that $h=gu$, where $h_{i,j}=0$ for all $(i,j)\in \widehat{K}$. 
\end{proof}

\subsection{Incidence geometry and Coset Geometry}\label{sec:back:subsec:incidence}
\subsubsection{Incidence Geometry}
    A 4-tuple $\Gamma = (X,I,*,t)$ is called an \textit{incidence system} if
    \begin{enumerate}
        \item $X$ is a set whose elements are called the \textit{elements} of $\Gamma$,
        \item $I$ is the set of types of $\Gamma$, 
        \item $*$ is a symmetric and reflexive relation (called the \textit{incidence relation}) on $X$, and
        \item $t$ is a map from $X$ to $I$, called the \textit{type map} of $\Gamma$, such that distinct elements $x,y \in X$ with $x * y$ satisfy $t(x) \neq t(y)$. 
    \end{enumerate}

The \textit{rank} of $\Gamma$ is the cardinality of the type set $I$.
For an element $x\in X$, we will call it an \emph{$i$-element} if $t(x)=i$, with $i\in I$.
In an incidence system $\Gamma$, a \textit{flag} is a set of pairwise incident elements. The rank of a flag, \text{rank}(F), is the cardinality of $F$, that is $\textnormal{rank}(F)=|F|$. The type of a flag $F$, $t(F)$, is the set of types of the elements of $F$. 
A \textit{chamber} is a flag of type $I$. An incidence system $\Gamma$ is an \textit{incidence geometry} if all its maximal flags are chambers.

The \textit{incidence graph} of $\Gamma$ is a graph with vertex set $X$ and where two distinct elements $x$ and $y$ are connected by an edge if and only if $x * y$. A incidence geometry $\Gamma$ of type $I=\{i,j\}$ is said to be a \emph{generalized digon} if all its elements of type $i$ are incident with all its elements of type $j$. Graph wise, a generalized digon gives a complete bipartite incidence graph.

Let $F$ be a flag of $\Gamma$. An element $x\in X$ is {\em incident} to $F$ if $x*y$ for all $y\in F$. The \textit{residue} of $\Gamma$ with respect to $F$, denoted by $\Gamma_F$, is the incidence system formed by all the elements of $\Gamma$ incident to $F$ but not in $F$. Whenever $\Gamma$ is a geometry, the \textit{rank} of the residue $\Gamma_F$ is defined to be rank$(\Gamma) - |F|$. If $\Gamma$ is of rank $n$, then the \textit{corank} of a flag $F$ is the difference between the rank of $\Gamma$ and the rank of $F$.
For a flag $F$ of type $t(F)$, we say that the residue $\Gamma_F$ is of type $I\setminus t(F)$.

An incidence geometry $\Gamma$ is \textit{connected} if its incidence graph is connected. It is \textit{residually connected} if all its residues of rank at least two are connected.
An incidence geometry is \textit{firm}, \textit{thin} or \textit{thick} if all its residues of flags of corank one contain, respectively, at least two elements,  exactly two elements, more than two elements.

Let $\Gamma = \Gamma(X,I,*,t)$ be an incidence geometry. A {\em correlation} of $\Gamma$ is a bijection $\phi$ of $X$ respecting the incidence relation $*$ and such that, for every $x,y \in X$, if $t(x) = t(y)$ then $t(\phi(x)) = t(\phi(y))$. If, moreover, $\phi$ fixes the types of every element (i.e $t(\phi(x)) = t(x)$ for all $x \in X$), then $\phi$ is said to be an {\em automorphism} of $\Gamma$. The group of all correlations of $\Gamma$ is denoted by $\Cor(\Gamma)$ and the automorphism group of $\Gamma$ is denoted by $\Aut(\Gamma)$. Remark that $\Aut(\Gamma)$ is a normal subgroup of $\Cor(\Gamma)$ since it is the kernel of the action of $\Cor(\Gamma)$ on $I$.

If $\Aut(\Gamma)$ is transitive on the set of flags of $\Gamma$ then we say that $\Aut(\Gamma)$ is {\em flag-transitive} on $\Gamma$. We say $\Aut(\Gamma)$ is {\em chamber-transitive} on $\Gamma$ if $\Aut(\Gamma)$ is transitive on the set of chambers of $\Gamma$. Note that, if $\Gamma$ is a geometry, being chamber-transitive is equivalent to being flag-transitive~\cite[Proposition 2.2]{hypertopes}. If moreover, the stabilizer of a chamber in $\Aut(\Gamma)$ is reduced to the identity, we say that $\Gamma$ is {\em simply transitive} or {\em regular}. If $G$ is any group together with a homomorphism $\psi \colon G \to \Aut(\Gamma)$, we also say that $G$ is flag-transitive, simply transitive or regular on $\Gamma$ if that is true of $\psi(G) \leq \Aut(\Gamma)$.
A \textit{regular hypertope} is a thin, residually connected and flag-transitive incidence geometry~\cite{hypertopes}.

\subsubsection{Coset Geometries}\label{sec:CosetGeom}
The main focus of this paper is in incidence geometries that are obtained from a group $G$ together with a set $(G_i)_{i \in I}$ of subgroups of $G$, as described by Tits in~\cite{Tits1957}. 
\begin{definition}[Coset incidence system]
    A \emph{coset incidence system} $\Gamma(G,(G_i)_{i \in I})$, denoted also as $(G,(G_i)_{i \in I})$, is the incidence system over the type set $I$ where:
    \begin{enumerate}
        \item The elements of type $i \in I$ are left cosets of the form $g G_i $, $g \in G$.
        \item The incidence relation is given by non-empty intersection. More precisely, the element $g G_i$ is incident to the element $k G_j $ if and only if $g G_i \cap k G_j \neq \emptyset$.
    \end{enumerate}
\end{definition}
We usually call the subgroups $G_i$ as \emph{(standard) maximal parabolic subgroups}. For $J\subseteq I$, we define \emph{(standard) parabolic subgroups} as $G_J=\cap_{i\in J} G_i$, and define $G_\emptyset=G$.
When $J=I\setminus\{i\}$, we say that $G_J=G_{I\setminus\{i\}}$ are the \emph{minimal parabolic subgroups}, and we will denote them as $G^i$. Finally, the subgroup $G_I=\cap_{i\in I} G_i$ is called the \emph{Borel subgroup} of $\Gamma$ and we will denote it as $B$. 
    
A coset incidence system that is an incidence geometry is called a \emph{coset geometry}.
Flag-transitive geometries can always be constructed as coset geometries.
The following theorem gives a group theoretical condition to check the flag-transitivity of $G$ on the coset incidence system $\Gamma(G,(G_i)_{i\in I})$. Notice that $G$ always naturally acts on the elements of $\Gamma(G,(G_i)_{i \in I})$ by left multiplication.

\begin{theorem}\cite[Proposition 2.2, adapted]{Piedade2026}\label{thm:cosetFT}
Let $\Gamma = (G,(G_i)_{i\in I})$ be a coset incidence system. Then $G$ is flag-transitive on $\Gamma$ if and only if for every $J,H,K \subseteq I$ we have $(G_J \cap G_H )(G_J \cap G_K ) = G_J \cap (G_H G_K )$.
Additionally, if $G$ is flag-transitive on $\Gamma$, then $\Gamma$ is a geometry.
\end{theorem}

The most interesting coset geometries $\Gamma = (G,(G_i)_{i\in I})$ are those for which $G$ acts flag-transitively on $\Gamma$. In that case, properties like residual connectedness, firmness, thinness and thickness can also be checked via group-theoretical conditions.

For a subset $J\subseteq I$, we define $G^J=\langle G^i\mid i\in J\rangle$. In other words, $G^J$ is the subgroup of $G$ generated by the minimal parabolic subgroups $G^i$, for $i\in J$. Notice that $G^I$ does not need to be necessarily equal to $G$. With these notations, the following theorem gives us conditions to guarantee residual connectedness of a coset geometry.

\begin{theorem}\cite[Corollary 1.8.13]{buekenhout2013diagram}\cite[Proposition 1.8.12]{buekenhout2013diagram}\cite[Lemma 2.5]{Piedade2026}
\label{thm:CosetRC}
    Suppose that $\Gamma = (G,(G_i)_{i\in I})$ is a coset incidence system over
the finite set $I$. The following statements are equivalent:
\begin{enumerate}
        \item for all $J \subset I$ such that $|I \setminus J| \geq 2$, we have
        $G_J = \langle G_{\{i\} \cup J} | i \in I \setminus J \rangle$;\label{lem:RCequivs:itm:RC1}
        \item for all $J \subset I$ and distinct $i,k\in I\setminus J$, we have 
        $G_J = \langle G_{\{i\} \cup J} , G_{\{k\} \cup J} \rangle$; \label{lem:RCequivs:itm:RC2}
        \item for every $J\subseteq I$, $G_{J} = G^{I\setminus J}$;\label{lem:RCequivs:itm:BottomUp}
        \item for every $J\subseteq I$, $G^J = G_{I\setminus J}$;\label{lem:RCequivs:itm:UpBottom}
        \item for every $J,K\subseteq I$, $G^J\cap G^K = G^{J\cap K}$;\label{lem:RCequivs:itm:Intersection}
    \end{enumerate}
Moreover, if G acts flag-transitively on $\Gamma$ and any of the above is true, then $\Gamma$ is a residually connected geometry.
\end{theorem}

Finally, we present here the group theoretical conditions for firmness, thinness, and thickness of flag-transitive coset geometries. The conditions below for thinness and thickness is not stated verbatim in \cite{buekenhout2013diagram}, but it is a trivial consequence of the result for firmness.

\begin{theorem}\label{thm:CosetFIRM}\cite[Corollary 1.8.15 (adapted)]{buekenhout2013diagram}
     Suppose that $\Gamma = (G,(G_i)_{i\in I})$ is a coset incidence system over
the finite set $I$ on which G acts flag-transitively. Then 
\begin{enumerate}
    \item $\Gamma$ is a firm geometry if and only if
    $G_{I\setminus \{j\} } \neq G_I$ for each $j \in I;$
    \item For each $j\in I$, let  $k_j := [G_{I\setminus \{j\} } : G_I]$. $\Gamma$ is a thin geometry (resp. thick geometry) if and only if $k_j=2$ (resp. $k_j>2$) for each $j\in I$.

\end{enumerate}
\end{theorem}

In some cases we will deal, the Borel subgroup is non-trivial. Hence, whenever the Borel is a normal subgroup of $G$ we can use the following theorem.

\begin{theorem}\cite{Tits1957}\label{thm:CGQuotient}
    Let $\Gamma=(G,(G_i)_{i\in I})$ be a coset geometry. If $K$ is the kernel of the action of $G$ on $\Gamma$
, then $(G,(G_i)_{i\in I}) \cong (G/K,(G_i/K)_{i\in I})$
\end{theorem}

Indeed, from~\cite{Tits1957}, we know that the kernel of the action of $G$ on $\Gamma$ is $K = \bigcap_{i\in I} (\bigcap_{g\in G} gG_ig^{-1})$, the biggest normal subgroup inside all of the maximal parabolic subgroups $G_i$. Hence, whenever $B$ is a normal subgroup, we get that $K = B$.

\subsubsection{Buekenhout Diagram}
Francis Buekenhout introduced in~\cite{buek} a diagram associated to incidence geometries. His insight was that, in many cases, the rank two residues of a geometry encode a lot of the global structure of the geometry. We will define this diagram only for flag-transitive geometries, as this ensures that all residues of type $\{i,j\}$, for $i,j\in I$, are isomorphic to each other. We will follow the definition of a diagram for a firm, residually connected, flag-transitive geometry over $I$, as found in~\cite{Buekenhout1996}.
\begin{definition}\label{def:BuekenhoutDiag}
  Let $\Gamma$ be a firm, residually connected, flag-transitive geometry over $I$.  The \emph{Buekenhout diagram} $\DD(\Gamma)$ is a complete graph whose vertex set is $I$ and with the following structure: 
\begin{itemize}
    \item to each vertex $i\in I$, we attach the order $o_i = |\Gamma_F|-1$, where $\Gamma_F$ is the residue of a flag of type $I\setminus\{i\}$;
    \item elements $i$ and $j$ of $I$ are not connected by an edge if, for all flags $F$ of type $I\setminus\{i,j\}$, the residues $\Gamma_F$ of type $\{i,j\}$ are always a generalized digon.
    \item otherwise, elements $i$ and $j$ of $I$ are connected by a labelled edge, with label $(d_{ij},g_{ij},d_{ji})$, where, for a residue $\Gamma_F$ of type $\{i,j\}$:
    \begin{enumerate}
        \item $g_{ij}$, the \emph{gonality}, is half of the girth of the incidence graph of $\Gamma_F$;
        \item $d_{ij}$ (resp. $d_{ji}$), the $i$-\emph{diameter} (resp. $j$-\emph{diameter}) is the greatest distance in the incidence graph of $\Gamma_F$ from an $i$-element (resp. $j$-element) to any other element in the incidence graph of $\Gamma_F$.
    \end{enumerate}  
\end{itemize}  
\end{definition}
Hence, for each rank 2 residue, we will have a diagram as follows.
$$\xymatrix@-1pc{*{\circ}\ar@{-}[rrr]^{d_{ij},g_{ij},d_{ji}}_(0.01){o_i}_(0.99){o_j} &&& *{\circ} }$$
Whenever $g_{ij}=d_{ij}=d_{ji}=k$, we say that this rank $2$ residue is a \emph{generalised $k$-gon} and, in this case, we ommit $d_{ij}$ and $d_{ji}$.

If $\Gamma$ is a firm, residually connected and flag-transitive coset geometry, we can determine the parameters $d_{ij}$, $g_{ij}$, $d_{ji}$ and $o_i$ using properties of the groups.
Consider $f\in \mathbb{N}_0$ and let $G^{[i,j]}_f$ denote the product of $f$ subgroups, alternating between $G^i$ and $G^j$, i.e. $G^{[i,j]}_f := G^iG^jG^i\ldots$, with $f$ factors. For example, we have that $G^{[i,j]}_3 = G^iG^jG^i$ and $G^{[j,i]}_2=G^jG^i$. By definition, we set $G^{[i,j]}_0=G^{[j,i]}_0=B$.
The following results are gathered from~\cite{Pasini1994}.
\begin{proposition}\cite[pp. 307]{Pasini1994}\label{prop:rank2_passini}
    Let $\Gamma=(G,(G_0,G_1))$ be a rank $2$ firm, residually connected, and flag-transitive coset geometry.
    Then,
    \begin{enumerate}
        \item $\Gamma$ is a generalized digon if and only if $G_0G_1=G_1G_0\ (=G)$;
        \item $d_{01}=f$ (resp. $d_{10}=f$) if and only if $G^{[0,1]}_f=G$ and $G^{[0,1]}_{f-1}\neq G$ (resp. $G^{[1,0]}_f=G$ and $G^{[1,0]}_{f-1}\neq G$);
        \item $g_{01}$ is the minimal positive integer $f$ such that $G^{[0,1]}_f\cap G^{[1,0]}_f \neq G^{[0,1]}_{f-1}\cup G^{[1,0]}_{f-1}$;
        \item for $i\in \{0,1\}$, $o_i=[G_i:B]-1$.
    \end{enumerate}
\end{proposition}
With the results above, we can describe all rank $2$ residues of any firm, residually connected, flag-transitive coset geometry of rank $n\geq 2$ and describe its Buekenhout diagram.

\subsection{Independent sets and Intersection Property}\label{sec:back:IP}

Let $G$ be a group generated by a finite set $S=\{s_i\mid i\in I\}$, i.e. $G = \langle S\rangle$. We say $S$ is an \emph{independent generating set} if, for all $i\in I$, $s_i\notin\langle s_j\mid j\in I\setminus\{i\}  \rangle$.
Additionally, we say that $S$ is a \emph{strong independent generating set} if for $J,K\subseteq I$, we have $$\langle s_j\mid j\in J\rangle\cap \langle s_j\mid j\in K\rangle = \langle s_j\mid j\in J\cap K\rangle.$$
In abstract polytope theory, the property of strong independence is usually called \emph{intersection property}, as it involves the intersection of the subgroups generated by subsets of $S$.

Here, we will introduce a weaker version of the intersection property where, instead of proving this property for all possible subsets of the generators, we prove it for partitions of the generators. This definition is crucial, as we will prove that the unitriangular groups, and some cases of pattern groups, satisfy this property for a particular partition of its generating set.

\begin{definition}[Intersection Property for a partition $P$]
Let $G=\langle S\rangle $ be a group generated by a finite set of generators $S$. Additionally, let $P=\{S_i\mid i\in I\}$ be a disjoint partition of $S$.
Then, we say $G$ respects the \emph{intersection property for a partition $P$} if for any $J,K\subseteq I$,
$$\langle \cup_{j\in J}S_j\rangle \cap \langle \cup_{j\in K}S_j\rangle = \langle \cup_{j\in J\cap K}S_j\rangle$$
\end{definition}

Notice that in the definition above, we get the usual intersection property by taking the partition $P$ to be made up of singletons, i.e. $P = \{\{s\}\mid s\in S\}$. Additionally, if a group $G$ generated by $S$ satisfies the intersection property, it is direct that it satisfies the intersection property for any disjoint partition $P$ of $S$.

In the following proposition, we will relate properties of residual connectedness given in Theorem~\ref{thm:CosetRC} with the intersection property for a partition. Hence, we will follow the notation of subgroups $G_J$, $G^i$, $G^J$ and $B$ introduced in Section~\ref{sec:CosetGeom}.

\begin{proposition}\label{prop:General_IP_equiv_RC_B=1}
    Let $G=\langle S\rangle $ be a group generated by a finite set of generators $S$, and $P=\{S_i\mid i\in I\}$ be a disjoint partition of $S$.
    Then, $G$ respects the intersection property for a partition $P$ if and only if \begin{enumerate}
        \item $G_J = G^{I\setminus J}$, for any $J\subseteq I$;
        \item $G^i=\langle S_i\rangle$ for all $i\in I$; and 
        \item $B = \langle 1\rangle$.
    \end{enumerate}
    with $G_{i}=\langle S_j\mid j\in I\setminus\{i\}\rangle$.
\end{proposition}
\begin{proof}
    The direct implication is trivial. Indeed, by the intersection property, we have that $G_J = \langle S_i \mid i\in I\setminus J\rangle$ and $G^i = G_{I\setminus\{i\}}= \langle S_i\rangle$. Hence, $G_J = \langle G^j \mid j\in I\setminus J\rangle = G^{I\setminus J}$. Moreover, the intersection property also forces $B = G_I=\langle 1\rangle$.

    Consider the reverse implication.
    Due to property (1), we can alternatively define parabolic subgroups by the minimal parabolic subgroups. Hence, considering $G^i = \langle S_i\rangle$, we have that for any $J\subseteq I$, we have that $G_{I\setminus J}= G^J = \langle G^j \mid j\in J\rangle = \langle S_j\mid j \in J\rangle$. In particular, note that the maximal parabolics match $G_i = G^{I\setminus\{i\}} = \langle G^j \mid j\in I\setminus\{i\}\rangle =\langle S_j\mid j\in I\setminus\{i\}\rangle$ 
    and 
    $G^I = \langle S_j\mid j\in I\rangle = \langle S\rangle = G$.
    By Theorem~\ref{thm:CosetRC}, we have that Property $(1)$ is equivalent to $G^J \cap G^K = G^{J\cap K}$.
    Now, for $J,K\subseteq I$ such that $J\cap K\neq \emptyset$ we have that 
    $G^J \cap G^K = \langle S_j \mid j \in J\rangle \cap \langle S_j\mid j \in K\rangle = G^{J\cap K}= \langle S_j\mid j\in J\cap K\rangle$.
    If $J\cap K=\emptyset$, then $G^J \cap G^K = \langle S_j \mid j \in J\rangle \cap \langle S_j\mid j \in K\rangle = G^{\emptyset} = G_I = B= \langle 1\rangle$. Hence, we have proved the intersection property for the partition $P$. 
    
\end{proof}

\section{Pattern geometry}\label{sec:PatternGeom}

We will first study the natural geometries associated with Pattern groups, which we will call the ``pattern geometry". As unitriangular groups are a special case of pattern groups where $\widehat{I} = \Delta_n$, any results obtained here apply also for geometries built using the unitriangular groups. Nevertheless, in Section~\ref{sec:UniTriaGeom}, we will explore a bit more the geometries obtained from unitriangular groups, as these geometries have particular properties.

In this section, we will define $U:= UT_{\widehat{I}}(n,\FF_q)$, the pattern group associated with $\widehat{I}$ over the field $\FF_q$. To be precise, we will take $n\geq 2$ and $q=p^m$, with $p$ prime, and $m\geq 1$ a positive integer.
Additionally, $\widehat{I}$ is a closed subset of $\Delta_n$, $\PP_{\widehat{I}}$ is the partial order set induced by $\widehat{I}$, and consider $I$ the set of the covering relations of $\PP_{\widehat{I}}$. Remind that for a set $K$, $\widehat{K}$ denotes the transitive closure of $K$.
Finally, we will denote by $T$ the basis of the field $\FF_q$ seen as a $\FF_p$-vector space.

For each covering relation $(i,j)\in I$, we will define the maximal parabolic subgroup 
$$U_{(i,j)}:=\langle e_{k,l}(t)\mid t\in T, (k,l)\in I\setminus\{(i,j)\}\rangle.$$
Note that, in the subgroup $U_{(i,j)}$, we can have any of the elements $e_{k,t}(t)$, with $t\in T$ and $(k,t)\in \widehat{I\setminus\{(i,j)\}}$, the closure of the set under Definition~\ref{def:closure}.
Hence, each subgroup $U_{(i,j)}$ can be equally described as $\langle e_{k,l}(t)\mid t\in T, (k,l)\in \widehat{I\setminus\{(i,j)\}}\rangle$. This description is essential whenever we will intersect maximal parabolic subgroups.

Let $\UU\TT_{\widehat{I}}(n,q):= (U,(U_{(i,j)})_{(i,j)\in I})$ be the coset incidence system defined by $U$ and its subgroups $U_{(i,j)}$ as defined above. Here, we will prove that this coset incidence system is a residually-connected, firm coset geometry where $U$ acts flag-transitively. 

\subsection{Properties of the Parabolic Subgroups of $\UU\TT_{\widehat{I}}(n,q)$}

First, we will give a first description of the remaining parabolic subgroups of $\UU\TT_{\widehat{I}}(n,q)$. In Proposition~\ref{prop:Patter_U_K_redefin}, we will show we can simplify greatly this description.

\begin{lemma}\label{lem:patter_U_K}
    Let $K\subseteq I$. Then the parabolic subgroups $U_K$ of $\UU\TT_{\widehat{I}}(n,q)$ are $$U_K := \cap_{(i,j)\in K} U_{(i,j)} = \left\langle e_{k,l}(t)\mid t\in T, (k,l)\in\bigcap_{(i,j)\in K} \widehat{I\setminus\{(i,j)\}}\right\rangle.$$
\end{lemma}
\begin{proof}
    Consider $u\in U_{K}$. Then we have that $u\in U_{(i,j)}$ for each $(i,j)\in K$, where we will take $U_{(i,j)}$ described as $\langle e_{k,l}(t)\mid t\in T, (k,l)\in \widehat{I\setminus\{(i,j)\}}\rangle$.
    Note that $\mathbf{U}_{(i,j)}=U_{(i,j)}-\mathbf{1}$ is a $\FF_p$-algebra with basis for $\mathbf{U}_{(i,j)}$
    $$\left\{e_{k,l}(t)-\mathbf{1}\mid t\in T,(k,l) \in\widehat{I\setminus\{(i,j)\}}\right\}$$
    i.e. all $(k,l)$ matrix entries for which there is a $(k,l)$-chain starting in $k$ and ending in $l$ in the poset $\PP_{\widehat{I}}$ after removing its $(i,j)$ covering relation.
    Hence, for $K\subseteq I$, the basis for the $(\cap_{(i,j)\in K} U_{(i,j)})-\mathbf{1}$ is
    $$\bigcap_{(i,j)\in K}\{e_{k,l}(t)-\mathbf{1}\mid t\in T,(k,l) \in\widehat{I\setminus\{(i,j)\}}\} = \{e_{k,l}(t)-\mathbf{1}\mid t\in T,(k,l) \in \cap_{(i,j)\in K}\widehat{I\setminus\{(i,j)\}}\}$$
    which implies that $U_K$ is generated by 
    $\langle e_{k,l}(t)\mid t\in T, (k,l)\in\cap_{(i,j)\in K}\widehat{I\setminus\{(i,j)\}}\rangle$.
\end{proof}

We would like to point out that the intersection of transitive closed sets is always a transitive closed set.
Hence, to avoid long and cumbersome notation, for $K\subseteq I$, we will for now on denote $\widehat{I_K} := \cap_{(i,j)\in K} \widehat{I\setminus\{(i,j)\}}$. We keep in this notation the ``hat" to remind that the resulting set is a transitive closed set.
Note that, under this notation, $\widehat{I_{\emptyset}} := \widehat{I}$, $\widehat{I_{\{(i,j)\}}} = \widehat{I\setminus\{(i,j)\}}$, and $\widehat{I_I} = \cap_{(i,j)\in I} \widehat{I\setminus\{(i,j)\}}$.
 From the point of view of chain of the poset $\PP_{\widehat{I}}$, an element $(k,l)\in \widehat{I_{\{(i,j)\}}}$ means that there exists a $(k,l)$-chain in $\PP_{\widehat{I}}$ which uses only covering relations $I\setminus\{(i,j)\}$. In general, for $K\subseteq I$, an element $(k,l)\in \widehat{I_{K}}$ means that, for each $(i,j)\in K$ there exists a $(k,l)$-chain using only covering relations $I\setminus\{(i,j)\}$.
 
 With this new notation, from now on, we will write $$U_{K} =  \langle e_{k,l}(t)\mid t\in T, (k,l)\in \widehat{I_{K}}\rangle.$$
One direct result of the lemma above is that all parabolic subgroups $U_K$ are themselves pattern groups, using the transitive closed set $\widehat{I_K}$, i.e. $U_K = UT_{\widehat{I_K}}(n,q)$.
Additionally, with the definition above, for each $(i,j)\in I$, we have that the minimal parabolic subgroups are $$U^{(i,j)}=\langle e_{k,l}(t)\mid t\in T, (k,l)\in \widehat{I_{I\setminus\{(i,j)\}}}\rangle.$$
Hence, with this, we can also describe also the Borel subgroup of $U$.

\begin{coro}\label{coro:Pattern_Borel}
The Borel subgroup $B$ of the coset incidence system $\UU\TT_{\widehat{I}}(n,q)$ is
    $$B := U_I= \bigcap_{(i,j)\in I} U_{(i,j)} = \langle e_{k,l}(t)\mid t\in T, (k,l)\in \widehat{I_I}\rangle.$$
\end{coro}
\begin{proof}
    This is a direct consequence of Lemma~\ref{lem:patter_U_K}.
\end{proof}

With the above, we have described all parabolic subgroups of $\UU\TT_{\widehat{I}}(n,q)$.
We will now prove that our parabolic subgroups respect the item $(2)$ of Theorem~\ref{thm:CosetRC}. In Section~\ref{sec:PatternIsFT}, we will prove that $U$ is flag-transitive in $\UU\TT_{\widehat{I}}(n,q)$. By Theorem~\ref{thm:CosetRC}, the following will give us that $\UU\TT_{\widehat{I}}(n,q)$ is residually connected geometry. 

\begin{proposition}\label{prop:pattern_RC1}
    Let $K\subseteq I$. Then for all $(a,b),(c,d)\in I\setminus K$ such that $(a,b)\neq (c,d)$, we have $U_K=\langle U_{K\cup\{(a,b)\}}, U_{K\cup\{(c,d)\}} \rangle$.
\end{proposition}
\begin{proof}
    Let $K\subseteq I$ and $(a,b),(c,d)\in I\setminus K$ such that $(a,b)\neq (c,d)$.
    By Lemma~\ref{lem:patter_U_K}, we have that $U_K = \langle e_{k,l}(t)\mid t\in T, (k,l)\in \widehat{I_K}\rangle$ and $$\langle U_{K\cup\{(a,b)\}}, U_{K\cup\{(c,d)\}}\rangle = \langle  e_{k,l}(t)\mid t\in T, (k,l)\in \widehat{I_{K\cup\{(a,b)\}}}\cup \widehat{I_{K\cup\{(c,d)\}}}\rangle$$
Note that we can rewrite $\widehat{I_{K\cup\{(a,b)\}}}\cup \widehat{I_{K\cup\{(c,d)\}}}$ as 
$\widehat{I_K}\cap (\widehat{I_{\{(a,b)\}}}\cup \widehat{I_{\{(c,d)\}}})$.

Since any $(k,l)\in \widehat{I_K}\cap (\widehat{I_{\{(a,b)\}}}\cup \widehat{I_{\{(c,d)\}}})$ belongs trivially to $\widehat{I_K}$, we have that all generators $e_{k,l}(t)$ of $\langle U_{K\cup \{(a,b)\}}, U_{K\cup \{(c,d)\}}\rangle$, for $t\in T$, are in $U_K$. Hence
$\langle U_{K\cup \{(a,b)\}}, U_{K\cup \{(c,d)\}}\rangle \subseteq U_K$.
To prove the other inclusion, consider a generator $e_{k,l}(t)\in U_K$, for $(k,l)\in \widehat{I_K}$ and any $t\in T$.

Our goal is to prove we can write the generator $e_{k,l}(t)\in U_K$, for $(k,l)\in \widehat{I_K}$, by taking only generators of $\langle U_{K\cup \{(a,b)\}}, U_{K\cup \{(c,d)\}}\rangle$.
Let us assume that $(k,l) \notin \widehat{I_K}\cap (\widehat{I_{\{(a,b)\}}}\cup \widehat{I_{\{(c,d)\}}})$, as otherwise we have trivially that the generator is already in $\langle U_{K\cup \{(a,b)\}}, U_{K\cup \{(c,d)\}}\rangle$.
In this case, since $(k,l)\in \widehat{I_K}$, we have $(k,l)\notin \widehat{I_{\{(a,b)\}}}\cup \widehat{I_{\{(c,d)\}}}=(\widehat{I\setminus\{(a,b)\}})\cup (\widehat{I\setminus\{(c,d)\}})$. Notice that $(k,l)$ cannot be a covering relation, as otherwise the only possibility would be that $(k,l)=(a,b)=(c,d)$, a contradiction.
Due to the equivalence between elements of $\widehat{I}$ and chains of $\PP_{\widehat{I}}$, 
this means that any $(k,l)$-chain of $\widehat{I}$ requires always both the covering relations $(a,b)$ and $(c,d)$ in their construction, otherwise $(k,l)$ would be found in either the closure of $I\setminus\{(a,b)\}$ or the closure of $I\setminus\{(c,d)\}$.
As $(a,b)$ and $(c,d)$ are always in a $(k,l)$-chain, we can assume, without lost of generality, that there is an ordering in $\PP_{\widehat{I}}$ such that $k \preccurlyeq a \prec b\preccurlyeq c \prec d \preccurlyeq l$.
Hence, considering this ordering, we know that
$(k,c)\in \widehat{I_{\{(c,d)\}}}$ and $(b,l)\in \widehat{I_{\{(a,b)\}}}$.
Since $(a,b)$ and $(c,d)$ are in all $(k,l)$-chains and $(k,l)\in \widehat{I_K}$, then we also have that $(k,c)$, $(a,b)$, $(c,d)$ and $(b,l)$ are in $\widehat{I_K}$.
As $(a,b)$ and $(c,d)$ are covering relations in $I\setminus K$, they have not been removed in $\widehat{I_K}$. Therefore, it is clear that $(a,b),(c,d)\in \widehat{I_K}$.
Furthermore, for each $(i,j)\in K$ we have $(k,l)\in \widehat{I_{\{(i,j)\}}}$.
Since both $(a,b)$ and $(c,d)$ are in all $(k,l)$-chains, then for each $(i,j)\in K$ there is always a $(k,a)$-chain, a $(d,l)$-chain and, provided $b\neq c$, a $(b,c)$-chain. This is equivalent to having $(k,a),(d,l),(b,c)\in \widehat{I_K}$. As $\widehat{I_K}$ is a closed set, both $(k,c)$ and $(b,l)$ will belong to $\widehat{I_K}$.
Finally, observe that both $(a,b)$ and $(c,d)$ are also elements of the union $\widehat{I_{\{(a,b)\}}}\cup \widehat{I_{\{(c,d)\}}}$.

So far, we have proven that $(k,c)$, $(b,l)$, $(a,b)$ and $(c,d)$ are in $\widehat{I_K}\cap (\widehat{I_{\{(a,b)\}}}\cup \widehat{I_{\{(c,d)\}}})$. 
With this, we will prove that we can write the generator $e_{k,l}(t)$ taking only generators indexed in $\widehat{I_K}\cap (\widehat{I_{\{(a,b)\}}}\cup \widehat{I_{\{(c,d)\}}})$.
Since $(k,c)$ and $(c,d)$ are elements of $\widehat{I_K}\cap (\widehat{I_{\{(a,b)\}}}\cup \widehat{I_{\{(c,d)\}}})$, using Proposition~\ref{prop:gens_props}(3), we can always write the element $e_{k,d}(t)\in \langle U_{K\cup\{(a,b)\}}, U_{K\cup\{(c,d)\}}\rangle $. 
Now, if $d=l$, we are done. Suppose then that $d\neq l$.
We claim that $(d,l)\in \widehat{I_K}\cap (\widehat{I_{\{(a,b)\}}}\cup \widehat{I_{\{(c,d)\}}})$. Indeed, we have previously proved that $(d,l)\in \widehat{I_K}$ and it is easy to check that $(d,l)\in \widehat{I_{\{(a,b)\}}}\cup \widehat{I_{\{(c,d)\}}}$.
Therefore, we have a generator $e_{d,l}(t)\in \langle U_{K\cup\{(a,b)\}}, U_{K\cup\{(c,d)\}}\rangle$ and, together with the generator $e_{k,d}(t)$, by Proposition~\ref{prop:gens_props}(3) we have the generator $e_{k,l}(t)\in \langle U_{K\cup\{(a,b)\}}, U_{K\cup\{(c,d)\}}\rangle$.
 We have therefore proved that all generators of $U_K$ are in $\langle U_{K\cup\{(a,b)\}} , U_{K\cup\{(c,d)\}}\rangle$, and therefore $U_K\subseteq \langle U_{K\cup\{(a,b)\}} , U_{K\cup\{(c,d)\}}\rangle$.
\end{proof}
 
By Theorem~\ref{thm:CosetRC}, the property of parabolic subgroups proved in Proposition~\ref{prop:pattern_RC1} is equivalent to proving that $U_K = U^{I\setminus K} = \langle U^{(i,j)}\mid (i,j)\in I\setminus K\rangle$.
This property allow us to greatly simplify the definition of our parabolic subgroup $U_K$ given by Lemma~\ref{lem:patter_U_K}. Indeed, the following proposition provides that simplification.

\begin{proposition}\label{prop:Patter_U_K_redefin}
Let $K\subseteq I$. Then
 $U_K = \langle e_{i,j}(t), B\mid t\in T,(i,j)\in I\setminus K\rangle$, where $B$ is the Borel subgroup.
\end{proposition}
\begin{proof}
Our main goal is to prove that all minimal parabolic subgroups $U^{(i,j)}$ can be written as $U^{(i,j)}= \langle e_{i,j}(t), B\mid t\in T\rangle$, for each $(i,j)\in I$.
We remind that, by Lemma~\ref{lem:patter_U_K}, we had  $U^{(i,j)}= \langle e_{k,l}(t)\mid t\in T, (k,l)\in \widehat{I_{I\setminus\{(i,j)\}}}\rangle$ and, by Corollary~\ref{coro:Pattern_Borel}, $B=\langle e_{k,l}(t)\mid t\in T, (k,l)\in \widehat{I_I}\rangle$.
We will introduce in this proof the following notation $\widehat{I^{\{(i,j)\}}}:=\widehat{I_{I\setminus\{(i,j)\}}}$.
First, notice that $\{(i,j)\}\cup \widehat{I_I} \subseteq \widehat{I^{\{(i,j)\}}}$, since $(i,j)\in \widehat{I^{\{(i,j)\}}}$ and $\widehat{I_I}\subseteq \widehat{I_{I\setminus\{(i,j)\}}}=\widehat{I^{\{(i,j)\}}}$.
Hence $\langle e_{i,j}(t), B\mid t\in T\rangle\subseteq U^{(i,j)}$.

To prove the reverse inclusion, we will prove that the generators of $U^{(i,j)}$ can be written in $\langle e_{i,j}(t), B\mid t\in T\rangle =\langle e_{k,l}(t)\mid t\in T, (k,l)\in \{(i,j)\}\cup \widehat{I_I}\rangle $.
In particular, take $(k,l)\in \widehat{I^{\{(i,j)\}}}$ and consider the generator $e_{k,l}(t)$ of $U^{(i,j)}$. 
Observe that 
\begin{align*}
    \{(i,j)\}\cup I_I &= \{(i,j)\}\cup \left(\bigcap_{(k,l)\in I} \widehat{I\setminus\{(k,l)\}}\right) \\ &=  \{(i,j)\}\cup (\widehat{I_{I\setminus\{(i,j)\}}}\cap \widehat{I\setminus\{(i,j)\}}) = \{(i,j)\}\cup (\widehat{I^{\{(i,j)\}}}\cap \widehat{I_{\{(i,j)\}}}).
\end{align*}
As $(k,l)\in \widehat{I^{\{(i,j)\}}}$, this means that, for each covering relation $(s,r)\in I\setminus\{(i,j)\}$, we can write a $(k,l)$-chain using only covering relations found in $I\setminus\{(s,r)\}$. If, for at least one $(s,r)\in I\setminus\{(i,j)\}$, we can write a $(k,l)$-chain without the covering relation $(i,j)$, meaning that we use only covering relations in $I\setminus\{(s,r),(i,j)\}$, then clearly $(k,l)\in \widehat{I\setminus\{(i,j)\}} = \widehat{I_{\{(i,j)\}}}$, hence $(k,l)\in \{(i,j)\}\cup (\widehat{I^{\{(i,j\})}}\cap \widehat{I_{\{(i,j\})}})$. Notice that an element $e_{k,l}(t)$, for $(k,l)$ in the conditions above, is a generator of $B$, since $(k,l)$ is in both $\widehat{I_{\{(i,j)\}}}$ and $\widehat{I^{\{(i,j)\}}}$, i.e. $(k,l)\in \widehat{I_I}$.

Now, let us suppose that, for all $(s,r)\in I\setminus\{(i,j)\}$, all $(k,l)$-chains written using covering relations $I\setminus\{(s,r)\}$ must always use the covering relation $(i,j)$. This also means that $k\preccurlyeq  i \prec j\preccurlyeq l$. We will assume that $k\neq i$ and $j\neq l$, otherwise the arguments are either trivial (the case where $(k,l)=(i,j)$) or one can adapt the following arguments accordingly.
Hence, for all $(s,r)\in I\setminus\{(i,j)\}$, we can always write $(k,i)$-chains and $(j,l)$-chains using the covering relations of $I\setminus\{(s,r)\}$, which implies that $(k,i),(j,l)\in \widehat{I^{\{(i,j)\}}}$. 
Finally, it is easy to check that $(k,i),(j,l)\in \widehat{I_{\{(i,j)\}}}$.
Hence, as all three $(k,i),(i,j)$, and $(j,l)$ are in $\{(i,j)\}\cup (\widehat{I^{\{(i,j)\}}}\cap \widehat{I_{\{(i,j)\}})}$, $(k,l)$ is an element in the transitive closure of $\{(i,j)\}\cup (\widehat{I^{\{(i,j)\}}}\cap \widehat{I_{\{(i,j)\}})}$, meaning that the generator $e_{k,l}(t)$ must belong to $ \langle e_{i,j}(t),B\mid t \in T\rangle$.
Therefore, we have that all minimal parabolics can be described as $U^{(i,j)}= \langle e_{i,j}(t), B\mid t\in T\rangle$.

By Proposition~\ref{prop:pattern_RC1} and Theorem~\ref{thm:CosetRC}, we have that $U_K = U^{I\setminus K}=\langle e_{i,j}(t), B\mid t\in T,(i,j)\in I\setminus K\rangle$, as wanted.
\end{proof}

This redefinition of the parabolic subgroups of $\UU\TT_{\widehat{I}}(n,q)$ allow us to study the intersection property of the parabolic subgroups of pattern groups associated with $\widehat{I}$.
Indeed, observe that the set of generators $\{e_{i,j}(t)\mid t\in T, (i,j)\in I\}$ can be partitioned by the covering relations $I$, with partition $P = \{S_{(i,j)}\mid (i,j)\in I\}$ where $S_{(i,j)} = \{e_{i,j}(t)\mid t\in T\}$. In the next proposition, we will give equivalent conditions for $UT_{\widehat{I}}(n,\FF_q)$ to have the intersection property for a partition $P$ as above. Particularly, one of this conditions derives from combinatorial properties of the poset $\PP_{\widehat{I}}$.

\begin{proposition}\label{prop:Pattern_B=1equiv}
The following are equivalent:
    \begin{enumerate}
        \item $U$ has the intersection property for the partition $P=\{S_{(i,j)}\mid (i,j)\in I\}$, where $S_{(i,j)} = \{e_{i,j}(t)\mid t\in T\}$;
        \item for all $K\subseteq I$, $U_K = \langle e_{i,j}(t)\mid t\in T, (i,j)\in I\setminus K\rangle$
        \item for all $(i,j)\in \widehat{I}$, $|C_{(i,j)}| = 1$
        \item $B = \langle\mathbf{1}\rangle$
    \end{enumerate}
\end{proposition}
\begin{proof}
    $\mathbf{(1\Leftrightarrow4)}$
    From Proposition~\ref{prop:General_IP_equiv_RC_B=1}, Proposition~\ref{prop:pattern_RC1} and Theorem~\ref{thm:CosetRC}, we get this equivalence. 
    
    $\mathbf{(2\Leftrightarrow4)}$ From Proposition~\ref{prop:Patter_U_K_redefin}, we get this equivalence.

    $\mathbf{(2\Rightarrow 3)}$
    Suppose that for all $K\subseteq I$, $U_K = \langle e_{i,j}(t)\mid t\in T, (i,j)\in I\setminus K\rangle$. First, note that if $(i,j)\in \widehat{I}$, then $i\prec j$ in $\PP_{\widehat{I}}$ and, hence, we know that there exists at least one $(i,j)$-chain. Therefore, $|C_{(i,j)}| \geq 1$. For a contradiction, suppose that there exists $(i,j)\in \widehat{I}$ such that $|C_{(i,j)}|\geq 2$. Then we have at least two distinct $(i,j)$-chains, so consider $w,k\in  C_{(i,j)}$ with $w\neq k$. Let us define $w=(w_1,w_2,\ldots,w_l)$ and $k=(k_1,k_2,\ldots,k_d)$, with $w_1=k_1=i$ and $w_l=k_d=j$, and let us assume that $l\leq d$ without lost of generality.
    Therefore, there exists at least one position $f\in \{2,\ldots,l-1\}$ such that $w_f\neq k_f$. We will consider $f$ to be the lowest position on the two sequences where these diverge for the first time. This implies that $(w_1,\ldots,w_{f-1})=(k_1,\ldots,k_{f-1})$, with $(w_{f-1},w_f), (w_{f-1},k_f)\in I$ and $(w_{f-1},w_f)\neq (w_{f-1},k_f)$.
    
    Suppose, by the hypothesis $(2)$, that every parabolic subgroup can be written as $U_K = \langle e_{i,j}(t)\mid t\in T, (i,j)\in I\setminus K\rangle$, for $K\subseteq I$.
    Note that Proposition~\ref{prop:pattern_RC1}, together with $(1\Leftrightarrow 4 \Leftrightarrow 2)$ gives us that $U^K =\langle e_{i,j}(t)\mid t\in T, (i,j)\in K \rangle$ and $ U^K\cap U^W = U^{K\cap W}$, for $K,W\subseteq I$.
    
    Now, since chains can be seen as ordered sequences of covering relations, we have that the chains $w$ and $k$ define two distinct subsets of $I$: $w$ defines the subset $W=\{(w_1,w_2),(w_2,w_3),\ldots,(w_{l-1},w_l)\}$ while $k$ defines the subset 
    $K=\{(k_1,k_2),(k_2,k_3),\ldots,(k_{d-1},k_d)\}$.

    Since $(w_{f-1},w_f)\neq (w_{f-1},k_f)$, then we have that neither $(w_{f-1},w_f)$ nor $(w_{f-1},k_f)$ are  elements of $K\cap W$. 
    Hence, we have that neither $e_{w_{f-1},w_f}(t)$ nor $e_{w_{f-1},k_f}(t)$ are generators of $U^{K\cap W}$, for any $t\in T$. Additionally, this means that we cannot write a $(i,j)$-chain using $K\cap W$.
    However, in both $U^K$ and $U^W$ we can write a word to represent the element of $e_{i,j}(t)$, which cannot happen in $U^{K\cap W}$, i.e. for any $t\in T$, $e_{i,j}(t)\in U^K\cap U^W$ but $e_{i,j}(t)\notin U^{K\cap W}$, contradicting the fact that $U^K\cap U^W=U^{K\cap W}$.

    $\mathbf{(3\Rightarrow 4)}$ Suppose that for all $(i,j)\in \widehat{I}$, $|C_{(i,j)}| = 1$. For a contradiction, suppose that $B\neq \langle \mathbf{1}\rangle$. This means that $\widehat{I_I}\neq \emptyset$ and, therefore, consider $(i,j)\in \widehat{I_I}$. Observe that this implies that, for all $(s,r)\in I$, one can always write a $(i,j)$-chain taking as covering relation those found in $I\setminus\{(s,r)\}$. However, as $|C_{(i,j)}|=1$, there is a unique $(i,j)$-chain, say $w = (w_1,w_2,\ldots,w_l)$, which cannot be written taking exclusively covering relations from, for example, $I\setminus\{(w_1,w_{2}\}$, a contradiction. Hence, $\widehat{I_I}=\emptyset$ and $B=\langle \mathbf{1}\rangle$.

\end{proof}

With this, we can see that whenever there is a $(i,j)\in I$ such that $|C_{(i,j)}|>1$, the  intersection property for this partition on the generators will always fail. 
Even when this intersection property fails, we still want to extract some information, particularly whenever the Borel subgroup $B$ is normal in $U$.

\begin{proposition}\label{prop:Pattern_BorelNormal}
    Suppose $B\neq \langle \mathbf{1}\rangle$. Then the following are equivalent:
    \begin{enumerate}
        \item $B$ is a normal subgroup of $U$;
        \item for each $(k,l)\in \widehat{I_I}$, and all $(i,j)\in \widehat{I}$ such that $i\preccurlyeq k$ and $l\preccurlyeq j$, we have $(i,j)\in \widehat{I_I}$.
    \end{enumerate}
\end{proposition}
\begin{proof}
$\mathbf{(1\Rightarrow 2)}$ Suppose $B$ is a normal subgroup of $U$ and let $(k,l)\in \widehat{I_I}$. If $k$ is a minimal element of $\PP_{\widehat{I}}$ and $l$ is a maximal element of $\PP_{\widehat{I}}$, we are done since there is no $i,j$ such that $i\prec k$ or $l \prec j$. If $k$ is not a minimal element of $\PP_{\widehat{I}}$, then there exists a covering relation $(i,k)\in I$.

We remind the reader that no elements of the covering relations are present in $\widehat{I_I}$.
Given the generator $e_{k,l}(t)\in B$, and since $B$ is a normal subgroup, we have  $$e_{i,k}(t)e_{k,l}(t)(e_{i,k}(t))^{-1} =e_{i,k}(t)e_{k,l}(t)e_{i,k}(-t) \in e_{i,k}(t)Be_{i,k}(-t) = B.$$ Hence, by writing $e_{i,k}(t)e_{k,l}(t)e_{i,k}(-t)$ as $(\mathbf{1} +f_{i,k}(t))(\mathbf{1} +f_{k,l}(t))( \mathbf{1}+f_{i,k}(-t))$ we have $\mathbf{1} + f_{k,l}(t) + f_{i,l}(t^2)\in B$. As $e_{k,l}(t)e_{i,l}(t^2) = \mathbf{1} + f_{k,l}(t) + f_{i,l}(t^2)$ and $e_{k,l}(t)\in B$, we have $e_{i,l}(t^2)\in B$ which is equivalent to having $(i,l)\in \widehat{I_I}$. If, instead, $l$ is not a maximal element of $\PP_{\widehat{I}}$, then there exists a covering relation $(l,j)\in I$, which, as above, results in $(k,j)\in \widehat{I_I}$.
We can repeat the above arguments until we get an element $(k',l')$ where $k'$ is a minimal element of $\PP_{\widehat{I}}$ and $l'$ is a maximal element of $\PP_{\widehat{I}}$. This repeated process gives us that all $(i,j) \in \widehat{I}$ such that $i\preccurlyeq k$ and $l\preccurlyeq j$ must be in $\widehat{I_I}$.

$\mathbf{(2\Rightarrow 1)}$
Suppose now that, for each $(k,l)\in \widehat{I_I}$, and all $(i,j)\in \widehat{I}$ such that $i\preccurlyeq k$ and $l\preccurlyeq j$, we have $(i,j)\in \widehat{I_I}$. We remind the reader that $B$ is a pattern group over $\widehat{I_I}$, hence all its non-zero entries are its main diagonal and entries $(i,j)\in \widehat{I_I}$.
Given an element of $g\in B$, one can write it as
$$g = \mathbf{1} + \sum_{(i,j)\in \widehat{I_I}} f_{i,j}(\alpha_{i,j}),$$
where, for each $(i,j)\in \widehat{I_I}$, $\alpha_{i,j}\in \FF_q$.
We are going to prove that, for any generator $e_{a,b}(t)$ of $U$, $e_{a,b}(t)\cdot g\cdot e_{a,b}(t)^{-1}=e_{a,b}(t)\cdot g\cdot e_{a,b}(-t)\in B$.
As $e_{a,b}(t) =\mathbf{1}+f_{a,b}(t)$, we have 
\begin{align*}e_{a,b}(t)\cdot g  &= \mathbf{1} + \sum_{(i,j)\in \widehat{I_I}} f_{i,j}(\alpha_{i,j})+ f_{a,b}(t) + \sum_{(i,j)\in \widehat{I_I}} \delta_{b,i}f_{a,j}(t\alpha_{i,j})\\
 &= \mathbf{1} + \sum_{(i,j)\in \widehat{I_I}} f_{i,j}(\alpha_{i,j})+ f_{a,b}(t) + \sum_{\substack{(i,j)\in \widehat{I_I}\\ b = i}} f_{a,j}(t\alpha_{i,j}).\end{align*}
Now, this results in $e_{a,b}(t)\cdot g\cdot e_{a,b}(-t)$ being equal to
$$\mathbf{1} + \sum_{(i,j)\in \widehat{I_I}} f_{i,j}(\alpha_{i,j})+ f_{a,b}(t) + \sum_{\substack{(i,j)\in \widehat{I_I}\\ b = i}} f_{a,j}(t\alpha_{i,j}) + f_{a,b}(-t) + \sum_{\substack{(i,j)\in \widehat{I_I}\\ a = j}} f_{i,b}(-t\alpha_{i,j}),$$
which is equivalent to
$$\mathbf{1} + \sum_{(i,j)\in \widehat{I_I}} f_{i,j}(\alpha_{i,j}) + \sum_{\substack{(i,j)\in \widehat{I_I}\\ b = i}} f_{a,j}(t\alpha_{i,j}) + \sum_{\substack{(i,j)\in \widehat{I_I}\\ a = j}} f_{i,b}(-t\alpha_{i,j}).$$
Now, for $(i,j)\in \widehat{I_I}$ such that $b=i$, this means that $i$ is not a minimal element of $\PP_{\widehat{I}}$. As $a \prec b = i$, then by our hypotheses, we have that $(a,j)\in \widehat{I_I}$. If we have $a = j$ for an $(i,j)\in \widehat{I_I}$, $j$ was not a maximal element of $\PP_{\widehat{I}}$ and, similarly, we have that $(i,b)\in\widehat{I_I}$.
Therefore, the matrix of $e_{a,b}(t)\cdot g\cdot e_{a,b}(-t)$ can only have non-zero entries exactly on the main diagonal and entries in $\widehat{I_I}$, meaning $e_{a,b}(t)\cdot g\cdot e_{a,b}(-t)\in B$.
As this happens for conjugation with any generator of $U$, we get that $B$ is a normal subgroup.
\end{proof}

 The proposition above might be a bit hard to understand conceptually, however it is more easily seen on the Hasse diagram of $\PP_{\widehat{I}}$. Below we will give some examples of sets of covering relations that provide a normal (or not) Borel subgroup.
\begin{exmp}
    \begin{enumerate}
        \item Let $I = \{(1,2),(1,3),(2,4),(3,4)\}$, with the following Hasse diagram.
        $$
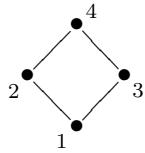
\xymatrix@-1pc{&*{\bullet}\ar@{-}[ld]\ar@{-}[rd]^(0.01){4} \\ 
        *{\bullet} && *{\bullet}\\
        & *{\bullet}\ar@{-}[lu]^(0.01){1}^(0.99){2}\ar@{-}[ru]_(0.99){3}}$$
        It is easy to calculate in this case that $\widehat{I_I} = \{(1,4)\}$. Indeed, a direct way of checking what are the elements of $\widehat{I_I}$ checking which paths on the poset are always present for each removal of a covering relation. In this case, it trivially satisfies the condition $(2)$ of Proposition~\ref{prop:Pattern_BorelNormal}. Hence, the Borel subgroup of this geometry is normal. 

        \item Alternatively, consider $I = \{(1,2),(1,3),(2,4),(3,4), (4,5)\}$, with the following Hasse diagram.
        $$
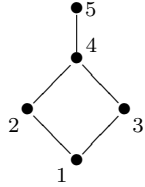
\xymatrix@-1pc{&*{\bullet}\ar@{-}[d]^(0.01){5}\\
        &*{\bullet}\ar@{-}[ld]\ar@{-}[rd]^(0.01){4} \\ 
        *{\bullet} && *{\bullet}\\
        & *{\bullet}\ar@{-}[lu]^(0.01){1}^(0.99){2}\ar@{-}[ru]_(0.99){3}}$$
        As in the example before, we have $\widehat{I_I} = \{(1,4)\}$. However, contrary to the previous case, we do not satisfy condition $(2)$ of Proposition~\ref{prop:Pattern_BorelNormal}. Indeed, $(1,4)\in \widehat{I_I}$ and $(1,5)\in \widehat{I}$ such that $4 \preccurlyeq 5$, however $(1,5)\notin \widehat{I_I}$. Hence, the Borel subgroup of this geometry is not normal. 
        
        \item Consider now $I = \{(1,2),(1,3), (1,4), (2,5), (3,5), (3,6),(4,6), (5,7), (6,7)\}$, with the following Hasse diagram.
        $$
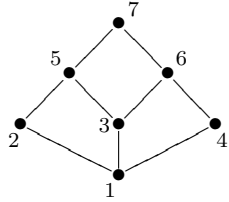
\xymatrix@-1pc{&& *{\bullet}\ar@{-}[rd]^(0.01){7}^(0.99){6}\ar@{-}[ld]_(0.99){5}\\
        & *{\bullet}\ar@{-}[rd]\ar@{-}[ld] && *{\bullet}\ar@{-}[rd]\ar@{-}[ld]\\ 
        *{\bullet} && *{\bullet} && *{\bullet}\\
        && *{\bullet}\ar@{-}[llu]^(0.01){1}^(0.99){2}\ar@{-}[u]^(0.99){3}\ar@{-}[rru]_(0.99){4}}$$
        By checking which paths are always present for each removal of a covering relation,  we can calculate that $\widehat{I_I} = \{(1,5),(1,6),(1,7),(3,7)\}$. Again, we can check the $\widehat{I_I}$ verifies condition $(2)$ of Proposition~\ref{prop:Pattern_BorelNormal} and, therefore, the borel subgroup, under this set of covering relations, will be normal.
        
        \item Finally, let $I$ be the set of covering relations in the following Hasse diagram.
        $$
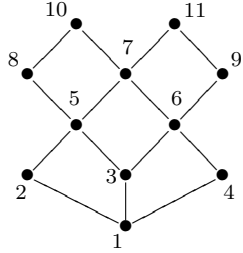
\xymatrix@-1pc{
        & *{\bullet}\ar@{-}[rd]^(0.75){7}\ar@{-}[ld]_(0.99){8}_(0.01){10} && *{\bullet}\ar@{-}[rd]^(0.99){9}^(0.01){11}\ar@{-}[ld]\\
        *{\bullet}\ar@{-}[rd] && *{\bullet}\ar@{-}[rd]^(0.75){6}\ar@{-}[ld]_(0.75){5} && *{\bullet}\ar@{-}[ld]\\
        & *{\bullet}\ar@{-}[rd]\ar@{-}[ld] && *{\bullet}\ar@{-}[rd]\ar@{-}[ld]\\ 
        *{\bullet} && *{\bullet} && *{\bullet}\\
        && *{\bullet}\ar@{-}[llu]^(0.01){1}^(0.99){2}\ar@{-}[u]^(0.99){3}\ar@{-}[rru]_(0.99){4}}$$
        We have $\widehat{I_I} = \{(1,5),(1,6),(1,7),(1,10), (1,11), (3,7), (3,10), (3,11), (5,10), (6,11)\}$. Contrary to the example before, the Borel subgroup of $\UU\TT_{\widehat{I}}(n,q)$ will not be normal. Indeed, we have $(5,10)\in \widehat{I_I}$ and $(2,10)\in \widehat{I}$ such that $2\preccurlyeq 5$ in this poset, however $(2,10)\notin \widehat{I_I}$. The same can be said about $(6,11)\in \widehat{I_I}$ and $(4,11)\notin \widehat{I_I}$.
        Therefore, the Borel subgroup in this case is not normal.
    \end{enumerate}

\end{exmp}

\subsection{Flag-Transitivity}\label{sec:PatternIsFT}
Finally, we will prove that $U$ acts flag-transitively in $\UU\TT_{\widehat{I}}(n,q)$. By Theorems~\ref{thm:CosetRC} and~\ref{thm:CosetFIRM}, we can prove the remaining properties like residual-connectedness, firmness, thinness and thickness from the properties found in the previous section on the  parabolic subgroups of $\UU\TT_{\widehat{I}}(n,q)$.

\begin{proposition}\label{prop:Pattern_FT}
    $U$ is flag-transitive on $\UU\TT_{\widehat{I}}(n,q)$.
\end{proposition}
\begin{proof}
    To check that $U$ is flag-transitive on $\UU\TT_{\widehat{I}}(n,q)$, following Theorem~\ref{thm:cosetFT}, we will prove that, for $J,K,L\subseteq I$, $(U_J \cap U_K)(U_J \cap U_L) = U_J\cap U_KU_L$. Note that the forward inclusion is trivial, hence we need only to prove the inverse inclusion.

    Consider $u\in U_J\cap U_K U_L$. We remind the reader that we can define a linear ordering on $\widehat{I}$ so that we can write uniquely each element of $U$ (see Section~\ref{sec:UTG_PG}). Consider then the ordering where $ \widehat{I_J}\setminus (\widehat{I_K}\cup \widehat{I_L})$ appear first, followed by elements in $\widehat{I_J} \cap (\widehat{I_K}\setminus \widehat{I_L})$, then elements in $\widehat{I_J}\cap \widehat{I_K}\cap \widehat{I_L}$, then elements from $\widehat{I_J}\cap (\widehat{I_L}\setminus \widehat{I_K})$, and finally elements of $\widehat{I}\setminus \widehat{I_J}$.
    Note that the first four sets are a partition of $\widehat{I_J}$, and that the five sets give a partition of $\widehat{I}$, hence we cover all possible elements of $\widehat{I}$.
    Additionally, even though we defined an ordering between these sets, we have not defined an ordering inside each of these. However, as any given linear ordering gives a unique word that describe each element of $U$, inside each of the above sets take a random linear ordering (the proof given here is independent on ordering chosen).
    Hence, consider $O_{\widehat{I}} = A \cup B \cup C \cup D \cup E$, where $A=O_{\widehat{I_J}\setminus(\widehat{I_K}\cup \widehat{I_L})}$, $B = O_{\widehat{I_J} \cap (\widehat{I_K}\setminus \widehat{I_L})}$, $C = O_{\widehat{I_J}\cap \widehat{I_K} \cap \widehat{I_L}}$, $D = O_{\widehat{I_J} \cap (\widehat{I_L}\setminus \widehat{I_K})}$, and $E = O_{\widehat{I}\setminus \widehat{I_J}}$.
    We can uniquely write any element $u\in U$ as the word
    $$ u = \left(\prod_{(i,j)\in A} e_{i,j}(t_{i,j}) \right) \left(\prod_{(i,j)\in B} e_{i,j}(t_{i,j}) \right) \left(\prod_{(i,j)\in C} e_{i,j}(t_{i,j}) \right) \left(\prod_{(i,j)\in D} e_{i,j}(t_{i,j}) \right) \left(\prod_{(i,j)\in E} e_{i,j}(t_{i,j}) \right),$$
    with $t_{i,j}\in \FF_q$ for $(i,j)\in \widehat{I}$.
    As $u\in U_J$ and since $U_J = UT_{\widehat{I_J}}(\FF_q)$, $u$ can be equally written with an ordering on $I_J$. Notice that $A\cup B\cup C\cup D$ already give a linear ordering of $\widehat{I_J}$.
    Therefore, the linear ordering of $U$ extend naturally to an ordering of $U_J$.
    As the elements of $E$ are not in $\widehat{I_J}$, in the word of $u$ we will have $t_{i,j}= 0$ for all $(i,j)\in E$.
    Additionally, as $u\in U_K U_L$, it can be written as a word with elements in $\widehat{I_K}$ and $\widehat{I_L}$. As $A$ is a linear ordering of elements $\widehat{I_J}\setminus (\widehat{I_K}\cup \widehat{I_L})$, we have that $u$ will have $t_{i,j}=0$ also for all $(i,j)\in A$.
    In the end, we have that a word for $u\in U_J\cap U_K U_L$ can be uniquely written with this ordering as
    $$ u = \left(\prod_{(i,j)\in B} e_{i,j}(t_{i,j}) \right) \left(\prod_{(i,j)\in C} e_{i,j}(t_{i,j}) \right) \left(\prod_{(i,j)\in D} e_{i,j}(t_{i,j}) \right).$$
    We have that  $\left(\prod_{(i,j)\in B} e_{i,j}(t_{i,j}) \right) \in U_J\cap U_K$ and 
    $\left(\prod_{(i,j)\in C} e_{i,j}(t_{i,j}) \right) \left(\prod_{(i,j)\in D} e_{i,j}(t_{i,j}) \right) \in U_J\cap U_L$. Additionally, note that 
    $\left(\prod_{(i,j)\in C} e_{i,j}(t_{i,j}) \right)$ belongs to both $U_J\cap U_K$ and $U_J\cap U_L$. Hence, $\left(\prod_{(i,j)\in C} e_{i,j}(t_{i,j}) \right) \in (U_J\cap U_K)(U_J\cap U_L)$, which means that  $u\in (U_J\cap U_K)(U_J\cap U_L)$, as wanted.
\end{proof}

Indeed, due to Propositions~\ref{prop:Cosets_representative} and~\ref{prop:Pattern_FT}, the set of flags of type $J\subseteq I$ $\UU\TT_{\widehat{I}}(n,q)$ are in bijection with a particular set of matrices 

\begin{coro}
    There is a bijection between flags of type $J\subseteq I$ and the set of matrices matrices $M_J = \{ m\in UT_{\widehat{I}}(n,q) \mid m_{i,j} = 0, \forall (i,j) \in \widehat{I_J}\}$.
\end{coro}
\begin{proof}
    Since $U$ is flag-transitive on $\UU\TT_{\widehat{I}}(n,q)$, for a given type $J\subseteq I$, we have that each flag of type $\UU\TT_{\widehat{I}}(n,q)$ can be represented as a coset of $U_J$ (for more details, see~\cite[Proposition 2.2]{Piedade2026}). Since $U_J \cong UT_{\widehat{I_J}}(n,q)$, by Proposition~\ref{prop:Cosets_representative}, for each flag represented by a coset $gU_J$, we can take the coset representative $g$ as a matrix $m$ such that $m_{i,j}=0$ for all $(i,j)\in \widehat{I_J}$. Finally, for the bijection, let $M_J = \{ m\in UT_{\widehat{I}}(n,q) \mid m_{i,j} = 0\; \forall (i,j) \in \widehat{I_J}\}$.
    Notice the size of $M_J$ is $q^{|\widehat{I}|-|\widehat{I_J}|}$ matrices in $M_J$, which is the exact number of cosets of $U_J$ in $U$ ($[U :U_J] = q^{|\widehat{I}|}/q^{|\widehat{I_J}|}$).
\end{proof}

\subsection{The Geometry of Pattern Groups}
 \begin{theorem}\label{thm:CG_of_Pattern}
    Consider a closed set $\widehat{I}\subseteq \Delta_n$, with covering relations $I$, and $q=p^m$, with $p$ prime and $m\geq 1$.
    Let $U:=UT_{\widehat{I}}(n,\FF_q)$ be the pattern group associated with $\widehat{I}$, and consider the coset incidence system $\UU\TT_{\widehat{I}}(n,q)=(U,(U_{(i,j)})_{(i,j)\in I})$ where $U_{(i,j)} := \langle e_{k,l}(t) \mid t\in T, (k,l)\in I\setminus\{(i,j)\} \rangle$. Then,
    \begin{enumerate}
        \item $U$ acts flag-transitively on $\UU\TT_{\widehat{I}}(n,q)$;
        \item $\UU\TT_{\widehat{I}}(n,q)$ is a residually connected coset geometry;
        \item $\UU\TT_{\widehat{I}}(n,q)$ is firm;
        \item $\UU\TT_{\widehat{I}}(n,q)$ is thin if and only if $q=2$ and $B\trianglelefteq U$;
        \item if $q\neq 2$ then $\UU\TT_{\widehat{I}}(n,q)$ is thick; moreover, if $B\trianglelefteq U$ and $\UU\TT_{\widehat{I}}(n,q)$ is thick, then $q\neq 2$;
        
    \end{enumerate}
\end{theorem}
\begin{proof}
    Items $(1)$ and $(2)$ are a direct consequence of Proposition~\ref{prop:Pattern_FT}, Proposition~\ref{prop:pattern_RC1} and Theorem~\ref{thm:CosetRC}.
    
    (3) Let us prove now that $\UU\TT_{\widehat{I}}(n,q)$ is firm.
    Suppose otherwise that $B=U^{(i,j)}$ for some $(i,j)\in J$. Note that $e_{i,j}(t)\in U^{(i,j)}$, for all $t\in T$, since $(i,j)\in \widehat{I^{\{(i,j)\}}}$. That would mean that, for all $t\in T$, $e_{i,j}(t)\in B$. However, since $B\subseteq U_{(i,j)}$, that would mean that $e_{i,j}(t)\in U_{(i,j)}$. Hence, $U_{(i,j)}=U$, which would contradict the fact that $\{e_{r,s}(t)\mid t\in T, (r,s)\in I\}$ is a minimal generating set of $U=UT_{\widehat{I}}(n,\FF_q)$ (Proposition~\ref{prop:Pattern_mingenset}). Hence, $\UU\TT_{\widehat{I}}(n,q)$ is firm.

    (4)
    First, suppose that $\UU\TT_{\widehat{I}}(n,q)$ is thin. Since $U$ is flag-transitive in $\UU\TT_{\widehat{I}}(n,q)$, we know that $|U^{(i,j)}:B|=2$ for all $(i,j)\in I$, i.e. $B$ is a normal subgroup for each minimal parabolic subgroup.
    By residual connectedness, we have that the subgroups $U^{(i,j)}=\langle e_{i,j}(t), B\mid t\in T\rangle$, $(i,j)\in I$, generate $U$, i.e. $U = \langle U^{(i,j)}\mid (i,j)\in I\rangle = \langle e_{i,j}(t), B\mid t\in T, (i,j)\in I\rangle$.
    Hence, since $B$ is a normal subgroup of every minimal parabolic subgroup $U^{(i,j)}$, then $B$ is also a normal subgroup of $U$.
    Additionally, as none of the covering relations $(i,j)\in I$ is in $\widehat{I_I}$, we have that for each $(i,j)\in I$, $\langle e_{i,j}(t)\mid t\in T\rangle\cap B$ is trivial. Therefore, $U^{(i,j)} \cong B\rtimes \langle e_{i,j}(t)\mid t\in T\rangle$, where $\langle e_{i,j}(t)|t\in T\rangle \cong (\ZZ/p\ZZ)^m$.
    With this, the only way $U^{(i,j)}/B$ has size two is if and only if $(\ZZ/p\ZZ)^m$ has size two, meaning $p=2$ and $m=1$.
    
    Consider now that $B\trianglelefteq U$ and $q=2$. Then, by Proposition~\ref{prop:Patter_U_K_redefin}, for all $(i,j)\in I$, we have $B\trianglelefteq U^{(i,j)}$. As before, we have that $U^{(i,j)} \cong B\rtimes \langle e_{i,j}(t)\mid t\in T\rangle$ and, since $q=2$, we have that $|U^{(i,j)}: B|=2$.
    
    (5)  
    Suppose that $q\neq 2$, meaning that either $p\neq 2$ or $m\neq 1$. Let $(i,j)\in I$. Note that, since $(i,j)$ is a covering relations, we know that $(i,j)\notin \widehat{I_I}$. As $B=\langle e_{k,l}(t)\mid t\in T, (k,l)\in \widehat{I_I}\rangle$, this is the pattern group over the closed set $\widehat{I_I}$, and therefore the entry $(i,j)$ must always be zero.
    
    Assume first that $p\neq 2$ and fix $t\in T$. Let us prove that $B$, $e_{i,j}(t)B$ and $e_{i,j}^2(t)B$ are distinct cosets, for $(i,j)\in I$.  Hence, we have trivially that $B \neq e_{i,j}(t)B$ and $e_{i,j}(t)B \neq e_{i,j}^2(t)B$, as otherwise $e_{i,j}(t)\in B$, a contradiction. Also, $B \neq e_{i,j}^2(t)B$, as that also implies that $e^2_{i,j}(t)=e_{i,j}(2t)\in B$. Taking $2^{-1}\in \ZZ_p$, we have $e^{2^{-1}}_{i,j}(2t)\in B$, meaning $e_{i,j}(t)\in B$, a contradiction. Hence, for this case, we have $|U^{(i,j)}:B|>2$  for every $(i,j)\in I$.
    Now, assume that $m\neq 1$ and, therefore, $|T|>1$. Choose distinct generators of the $\FF_p$-basis $t,t'\in T$. To prove that $B$, $e_{i,j}(t)B$ and $e_{i,j}(t')B$ are distinct cosets, notice that $B\neq e_{i,j}(t)B$ and $B \neq e_{i,j}(t')B$, as that would imply that either $e_{i,j}(t)$ or $e_{i,j}(t')$ are in $B$, a contradiction. Also, $e_{i,j}(t)B \neq e_{i,j}(t')B$, as that would imply that $e_{i,j}(t-t')\in B$, i.e. there would be a matrix with non-zero $(i,j)$-entry in $B$, a contradiction.
    Again, we have that $|U^{(i,j)}:B|>2$.
    Hence, since $U$ is flag-transitive in $\UU\TT_{\widehat{I}}(n,q)$, by Theorem~\ref{thm:CosetFIRM} then $\UU\TT_{\widehat{I}}(n,q)$ is thick.
    
    Now, let $B\trianglelefteq U$. Suppose that $\UU\TT_{\widehat{I}}(n,q)$ is thick. Then, for any $(i,j)\in I$, $|U^{(i,j)}:B|>2$. For a contradiction, suppose that $q=2$. This implies that all $e_{i,j}^2(t)=\textbf{1}$ and that
    all minimal parabolic subgroups are $U^{(i,j)}=\langle e_{i,j}(t), B\rangle$, for a fixed $t$ such that $T=\{t\}$. Notice that since $B$ is normal in $U$, we have that $U^{(i,j)}/B \cong (\ZZ/2\ZZ)$. However, this would imply that $|U^{(i,j)}:B|=2$, a contradiction with the fact that it is thick. Hence, either $p\neq 2$ or $m\neq 1$.

\end{proof}

Notice that requiring that $q=2$ is not enough to given thinness. Indeed, take the case where $n=5$ and $I=\{ (1, 2 ), ( 3, 4 ), ( 4, 5 ), ( 1, 3 ), ( 2, 4 )\}$.
We can calculate that $B=\langle e_{1,4}(1)\rangle$, a subgroup of order $2$, and $U^{(4,5)}=\langle e_{4,5}(1),e_{1,4}(1)\rangle$.
By Proposition~\ref{prop:gens_props}(3), we have that $e_{1,5}(1)\in U^{(4,5)}$. Indeed, we can calculate that $|U^{(4,5)}|=8$, and therefore $|U^{(4,5)}:B|=8/2 = 4$, failing thinness by Theorem~\ref{thm:CosetFIRM}.

\begin{coro}\label{coro:Pattern_with_q=2_hyper}
    $\UU\TT_{\widehat{I}}(n,q)$ is a regular hypertope if and only if $q=2$ and $B\trianglelefteq UT_{\widehat{I}}(n,\FF_q)$.
\end{coro}
\begin{proof}
This follows from the definition of a regular hypertope, and the fact that $\UU\TT_{\widehat{I}}(n,q)$ is a residually connected thin coset geometry, with $U$ acting flag-transitively on $\UU\TT_{\widehat{I}}(n,q)$ when $q=2$ and $B\trianglelefteq UT_{\widehat{J}}(\FF_q)$.
\end{proof}

Notice that, in the above result, we are not requiring $B$ to be trivial. Indeed, in regular hypertopes, the Borel subgroup (i.e. the stabilizer of a chamber) is taken to be trivial. In the case where $B\neq \langle \textbf{1}\rangle$ (and hence, in the regular hypertope case, it is a normal subgroup of $UT_{\widehat{I}}(n,\FF_q)$), the coset geometry $\UU\TT_{\widehat{I}}(n,q) \cong (U/B, ( U_{(i,j)}/B)_{(i,j)\in I})$, following Theorem~\ref{thm:CGQuotient}.

\section{Unitriangular geometry}\label{sec:UniTriaGeom}

In the previous section, we have looked at subgroups of the unitriangular group $UT(n,\FF_q)$ by taking closed subsets of $\Delta_n$.
In fact, the unitriangular group $UT(n,\FF_q)$ is exactly the pattern group obtained by taking the closed set to be $\widehat{I}=\Delta_n$. Hence, many of the properties of the geometry associated with this group, $\UU\TT_{\Delta_n}(n,q)$ have been given in the previous section. 
However, since we are using $\widehat{I}=\Delta_n$, some nice properties arise that are worth studying.

Let us denote $\UU\TT(n,q):=\UU\TT_{\Delta_n}(n,q)$.
When $\widehat{I}=\Delta_n$, the set of covering relations is $I = \{(1,2),(2,3),$ $\ldots,(n-1,n)\}$, meaning that $\PP_{\widehat{I}}$ is a total order, or equivalently, its Hasse diagram is a linear graph.
Since all elements of $I$ are of the form $(i,i+1)$, with $1\leq i\leq n-1$, the maximal parabolic subgroups of $\UU\TT(n,q)$ are $U_{(i,i+1)}=\langle e_{j,j+1}(t)\mid t\in T, (j,j+1)\in I\setminus\{(i,i+1)\}\rangle$.
To simplify notation, consider $e_i(t):= e_{i,i+1}(t)$ and $U_i := U_{(i,i+1)}$ which, in turn, makes it possible for us to simplify $I=\{1,2,\ldots,n-1\}$.
Therefore, given $U:= UT(n,\FF_q)$, we can simplify $\UU\TT(n,q)$ as the coset geometry $(U,(U_i)_{i\in I})$. From Theorem~\ref{thm:CG_of_Pattern}, we have that $\UU\TT(n,q)$ is a firm, residually connected coset geometry with $U$ acts flag-transitively.

Another direct result we can obtain from pattern groups is that the set of generators of $UT(n,\FF_q)$ always satisfies the intersection property on a partition defined by $\{1,\ldots,n-1\}$.

\begin{proposition}\label{prop:UniT_IP}
Let $K,L\subseteq I$. Then $$\langle e_{i}(t)\mid t\in T, i\in K \rangle \cap \langle e_{i}(t)\mid t\in T, i\in L \rangle = \langle e_{i}(t)\mid t\in T, i\in K\cap L \rangle.$$
\end{proposition}
\begin{proof}
    First, note that the poset $\PP_{\Delta_n}$ defined by the covering relations $\{(1,2),(2,3),\ldots,(n-1,n)\}$ is a total order. This means that, for any $i,j\in \{1,\ldots,n\}$, such that $i\prec j$, we have that $|C_{(i,j)}|=1$. By Proposition~\ref{prop:Pattern_B=1equiv}, we get the result of the statement of the proposition.
\end{proof}

Hence, $UT(n,\FF_q)$ satisfies the generalized intersection property with partition $S = \{S_i\mid i\in I\}$ where $S_i = \{e_i(t)\mid t\in T\}$.

By Proposition~\ref{prop:gens_props}(1), we have that the generators $e_i(t)$ and $e_j(t')$ commute if $|i-j|>1$. Due to this and the total order of the poset, we can say much more about the structural characteristics of the $\UU\TT(n,q)$ geometries. Our goal in this section is to prove that the Buekenhout diagram associated with $\UU\TT(n,q)$ is linear and describe each consecutive rank 2 residues. In particular, we will prove that, whenever $q = 2$, $\UU\TT(n,q)$ gives a new infinite family of abstract regular polytopes.

\subsection{Buekenhout diagram}
First, let us prove that the Buekenhout diagram of $\UU\TT(n,q)$ is indeed linear.

\begin{proposition}\label{prop:UTnq_lin_diag}
    Let $\DD$ be the Buekenhout diagram of $\UU\TT(n,q)$. Then, $\DD$ is a linear diagram.
\end{proposition}
\begin{proof}
Let $\DD$ be the diagram of the coset geometry $\UU\TT(n,q)$. For simplicity, we will denote $\UU\TT(n,q)$ as $\UU$. 
First, note that there is a natural ordering of $I$, i.e. $I= \{1,2,\ldots,n-1\}$.
By Definition~\ref{def:BuekenhoutDiag}, we have that two types $i,j\in I$ are not connected by an edge if the residue of type $\{i,j\}$ is a generalized digon. In order to prove linearity of the diagram $D$, we will prove that, whenever $\{i,j\}$ are not consecutive types, the residue of type $\{i,j\}$ is a generalized digon, while whenever $\{i,j\}$ are consecutive types, these will be connected in $D$. Given the flag $\{U_k, k\in I\setminus\{i,j\}\}$ of $\UU$, we have that the residue of this flag is $\UU_{i,j} :=(U^{\{i,j\}},(U^i,U^j))$. By the intersection property proved above, we have that $U^i=\langle e_{i}(t)\mid t\in T\rangle$, $U^j=\langle e_{j}(t)\mid t\in T\rangle$ and $U^{\{i,j\}}=\langle e_{k}(t)\mid t\in T, k\in\{i,j\}\rangle$.

Let $i,j\in I$ such that $|i-j|>1$.  
By Proposition~\ref{prop:gens_props}(1), we have that, for any $\alpha,\beta\in \FF_q$, $e_{i}(\alpha)$ commutes with $e_{j}(\beta)$. Therefore, any $g\in U^{\{i,j\}}$ can be written as $g=e_{i}(\alpha)e_j(\beta)=e_j(\beta)e_{i}(\alpha)$, with $\alpha,\beta\in \FF_q$. Hence, since $U^{\{i,j\}} = U^i U^j=U^j U^i$, by Proposition~\ref{prop:rank2_passini}(1) we have that $\UU_{i,j}$ is a generalized digon.

Consider now the case where $|i-j|=1$ and, without loss of generality, let $j=i+1$ with $i\in \{1,\ldots,n-2\}$. 
We will now prove that $\UU_{i,i+1}=(U^{\{i,i+1\}},(U^i,U^{i+1}))$ is not a generalized digon. From Proposition~\ref{prop:gens_props}(1),  we have $e_{i}(t)e_{i+1}(t')\neq e_{i+1}(t')e_{i}(t)$, for any $\alpha,\beta\in \FF_q\setminus\{0\}$.
Hence, contrary to the case above, consecutive generators do not commute. 
We remind also the reader that $o(e_i(t))=p$, for any $t\in T$.
Let $T=\{t_1,\ldots,t_m\}$. 
We will prove now that the subgroups $U^i \cap e_{i+1}(1)e_{i}(1)U^{i+1}=\emptyset$.
For a contradiction, suppose there is an element of $U^i$ that belongs to $e_{i+1}(1)e_{i}(1)U^{i+1}$, i.e. $\exists (k_1,\ldots,k_m),(l_1,\ldots,l_m)\in (\mathbb{N}_0)^m$ such that $0\leq k_j,l_j<p$ and 
$$\prod _{j=1}^m e_{i}(t_j)^{k_j} =e_{i+1}(1)e_{i}(1)\left(\prod _{j=1}^m e_{i+1}(t_j)^{l_j}\right)$$
 which, by Proposition~\ref{prop:gens_props}(2) and (4), is equivalent to 
\begin{equation}\label{eq:lin_diag_eq} e_{i}\left(\sum_{j=1}^mk_jt_j\right)e_{i+1}\left(-\sum _{j=1}^ml_jt_j\right)=e_{i+1}(1)e_{i}(1).  
\end{equation}
Let us denote $\overline{k} = \sum_{j=1}^mk_jt_j$ and $\overline{l}=\sum _{j=1}^ml_jt_j$ and 
recall that we can write $e_i(t)=\mathbf{1}+f_{i,i+1}(t)$.
Therefore, we have that 
$$
e_{i}(\overline{k})e_{i+1}(-\overline{l})=\mathbf{1}+f_{i,i+1}(\overline{k})+f_{i+1,i+2}(-\overline{l})
+f_{i,i+2}(-\overline{k}\ \overline{l})
$$
which, by Equation~\ref{eq:lin_diag_eq}, is equal to
$$
e_{i+1}(1)e_{i}(1) =\mathbf{1}+f_{i+1,i+2}\left(1\right)+f_{i,i+1}\left(1\right).
$$
Therefore, this means that $\overline{k}=1=-\overline{l}$
and $-\overline{k}\ \overline{l}=0$ which is a contradiction. Hence, $U^i\cap e_{i+1}(1)e_{i}(1)U^{i+1}=\emptyset$, meaning that there is an element of type $i$ that is not incident with one element of type $i+1$, i.e. $\mathcal{U}_{i,i+1}$ is not a generalized digon.

Both these results imply that $\DD$, with vertex set $I$ in the natural order $\{1,\ldots,n-1\}$, is a linear Buekenhout diagram.

\end{proof}

To fully describe the Buekenhout diagram, we still need to describe every rank 2 residue of type $\{i,i+1\}$ which are not generalized digons. In the next proposition, we will determine the $i$-diameter $d_{i,i+1}$, the $i+1$-diameter $d_{i+1,i}$, and the gonality $g_{i,i+1}$ of these rank $2$ residues.

\begin{proposition}\label{prop:diag_props}
    For $i\in I$ we have that:
    \begin{enumerate}
        \item $ U^{[i,i+1]}_4 = U^{[i+1,i]}_4$;
        \item $U^{[i,i+1]}_3 \neq U^{[i,i+1]}_4$ and $U^{[i+1,i]}_3 \neq U^{[i+1,i]}_4$.
        \item Let $f$ be the minimal positive integer such that $U^{[i,i+1]}_f\cap U^{[i+1,i]}_f\neq U^{[i,i+1]}_{f-1}\cup U^{[i+1,i]}_{f-1}$; if $q\neq 2$ then $f=3$, otherwise, $f=4$.
    \end{enumerate}
\end{proposition}
\begin{proof}
    $\textbf{(1)}$ Let $x\in U^{[i+1,i]}_4 = U^{i+1}U^iU^{i+1}U^i$.
    We have that $U^j=\langle e_j(t)\mid t\in T\rangle$, for any $j\in I$. This means that any element of $U^j$ can be written as $e_j(\alpha)$, with $\alpha\in \FF_q$.
    First, note that if $x$ is in either of the following sets 
    \begin{align*}
        \{&e_i(\alpha_1),\ e_{i+1}(\alpha_1)\mid \alpha_1\in \FF_q\}\\
        \{&e_{i}(\alpha_1)e_{i+1}(\alpha_2),\ e_{i+1}(\alpha_1)e_i(\alpha_2) \mid \alpha_1,\alpha_2\in\FF_q\}\\
        \{&e_{i}(\alpha_1)e_{i+1}(\alpha_2)e_{i}(\alpha_3),\ e_{i+1}(\alpha_1)e_{i}(\alpha_2)e_{i+1}(\alpha_3)\mid \alpha_1,\alpha_2,\alpha_3\in \FF_q\}
    \end{align*}
    we can easily check that $x\in U_4^{[i,i+1]}$.
     Hence, consider then $x=e_{i+1}(\alpha_1)e_{i}(\alpha_2)e_{i+1}(\alpha_3)e_i(\alpha_4)$, where $\alpha_k\in \FF_q\setminus\{0\}$ for all $k=1,2,3,4$.
    By writting an element $e_{j}(\alpha)$ as $(\mathbf{1} + f_{j,j+1}(\alpha))$, we have that $e_{i+1}(\alpha_1)e_{i}(\alpha_2)=\mathbf{1}+f_{i,i+1}(\alpha_2)+f_{i+1,i+2}(\alpha_1)$ and $e_{i+1}(\alpha_3)e_{i}(\alpha_4)=\mathbf{1}+f_{i,i+1}(\alpha_4)+f_{i+1,i+2}(\alpha_3)$. First, we will assume that $\alpha_1+\alpha_3\neq 0$. Therefore, we have that 
    \begin{align*}
        e_{i+1}(\alpha_1)e_{i}(\alpha_2)e_{i+1}(\alpha_3)e_i(\alpha_4) &= (\mathbf{1}+f_{i,i+1}(\alpha_2)+f_{i+1,i+2}(\alpha_1))(\mathbf{1}+f_{i,i+1}(\alpha_4)+f_{i+1,i+2}(\alpha_3))\\
        &= \mathbf{1} + f_{i,i+1}(\alpha_2+\alpha_4) + f_{i+1,i+2}(\alpha_1+\alpha_3) + f_{i,i+2}(\alpha_2\alpha_3)\\
        & = e_{i}(\alpha_2+\alpha_4)e_{i+1}(\alpha_1+\alpha_3)(\mathbf{1} - f_{i,i+2}(A))  
    \end{align*}
    with $A = \alpha_2\alpha_1 + \alpha_4\alpha_1+\alpha_4\alpha_3$.
    Then, we have that 
    \begin{align*}
        e_{i,i+2}(-A) = \mathbf{1} - f_{i,i+2}(A) &= e_{i+1}(-\alpha_1-\alpha_3)e_i(-A(\alpha_1+\alpha_3)^{-1}) e_{i+1}(\alpha_1+\alpha_3)e_i(A(\alpha_1+\alpha_3)^{-1})
    \end{align*}
    Hence, as $e_{i+1}(-\alpha_1-\alpha_3) = e_{i+1}(\alpha_1+\alpha_3)^{-1}$, we have
\begin{align*}
        x&= e_{i}(\alpha_2+\alpha_4)e_{i+1}(\alpha_1+\alpha_3)(\mathbf{1} - f_{i,i+2}(A))\\
        &= e_{i}(\alpha_2+\alpha_4)e_{i+1}(\alpha_1+\alpha_3)e_{i+1}(-\alpha_1-\alpha_3)e_i(-A(\alpha_1+\alpha_3)^{-1}) e_{i+1}(\alpha_1+\alpha_3)e_i(A(\alpha_1+\alpha_3)^{-1})\\
        &= e_{i}(\alpha_2+\alpha_4-A(\alpha_1+\alpha_3)^{-1}) e_{i+1}(\alpha_1+\alpha_3)e_i(A(\alpha_1+\alpha_3)^{-1})
    \end{align*}
which belongs to $U_4^{[i,i+1]}$.
Assume now that $\alpha_1+\alpha_3=0$. In this case we have that 
\begin{align*}
e_{i+1}(\alpha_1)e_{i}(\alpha_2)e_{i+1}(\alpha_3)e_i(\alpha_4) &= (\mathbf{1}+f_i(\alpha_2)+f_{i+1}(\alpha_1))(\mathbf{1}+f_i(\alpha_4)+f_{i+1}(\alpha_3))\\
        &= \mathbf{1} + f_i(\alpha_2+\alpha_4) + f_{i,i+2}(\alpha_2\alpha_3)\\
        & = e_{i}(\alpha_2+\alpha_4)(\mathbf{1} + f_{i,i+2}(\alpha_2\alpha_3)).  
    \end{align*}
    As before, note that we can rewrite $e_{i,i+2}(\alpha_2\alpha_3)= \mathbf{1} + f_{i,i+2}(\alpha_2\alpha_3)$ as the following product $$e_{i}(1)e_{i+1}(\alpha_2\alpha_3)e_i(-1)e_{i+1}(-\alpha_2\alpha_3).$$
    Therefore, we have
\begin{align*}
        x&= e_{i}(\alpha_2+\alpha_4)(\mathbf{1} + f_{i,i+2}(\alpha_2\alpha_3))\\
        &= e_{i}(\alpha_2+\alpha_4)e_{i}(1)e_{i+1}(\alpha_2\alpha_3)e_i(-1)e_{i+1}(-\alpha_2\alpha_3)\\
        &= e_{i}(\alpha_2+\alpha_4+1)e_{i+1}(\alpha_2\alpha_3)e_i(-1)e_{i+1}(-\alpha_2\alpha_3)
    \end{align*}
which also belongs to $U_4^{[i,i+1]}$.
With this, we have proven that $U^{[i+1,i]}_4\subseteq U^{[i,i+1]}_4$
We can use equivalent arguments to prove that $U^{[i,i+1]}_4\subseteq U^{[i+1,i]}_4$.

    $\textbf{(2)}$ Now, clearly $U_3^{[i,i+1]}\subset U_4^{[i,i+1]}$. Consider $x\in U_4^{[i,i+1]}$ such that $x=e_{i,i+2}(a)$ and $a\in \FF_q\setminus \{ 0\}$. Indeed, we know that $x\in U_4^{[i,i+1]}$ by Proposition~\ref{prop:gens_props}(3).
    Suppose we could write $x$ as a word in $U_3^{[i,i+1]}$, i.e. $x = e_i(\alpha_1)e_{i+1}(\alpha_2)e_i(\alpha_3)$.
    Therefore, we can express $x$ as 
    \begin{align*}
        x&= \mathbf{1} + f_i(\alpha_1+\alpha_3)+f_{i+1}(\alpha_2)+f_{i,i+2}(\alpha_1\alpha_2).
    \end{align*}
    In order for $x=e_{i,i+2}(a)$, we need $\alpha_1\alpha_2=a$, $\alpha_1+\alpha_3=0$ and $\alpha_2=0$, a contradiction.
    Hence $x\notin U_3^{[i,i+1]}$.
    A similar argument proves the case $U_3^{[i+1,i]}\neq U_4^{[i+1,i]}$

    $\textbf{(3)}$ 
    First, note that the reverse inclusion is always true, as  $U_{f-1}^{[i,i+1]}$ and $U_{f-1}^{[i+1,i]}$ are clearly subsets of both $U_{f}^{[i,i+1]}$ and $U_{f}^{[i+1,i]}$. 
    
    We will first prove that for any pair $\{i,i+1\}$ ($i=1,\ldots,n-1$) and $f=1,2$, we have $U^{[i,i+1]}_f\cap U^{[i+1,i]}_f = U^{[i,i+1]}_{f-1}\cup U^{[i+1,i]}_{f-1}$.
    For $f=1$, we have that $U^{[i,i+1]}_1\cap U_1^{[i+1,i]} = U^i\cap U^{i+1} = \langle \mathbf{1}\rangle$ by Proposition~\ref{prop:UniT_IP}. As, by definition, $U_0^{[k,l]}=B$ and $B=\langle \mathbf{1}\rangle$, we have $U_0^{[i,i+1]}=\langle \mathbf{1}\rangle = U_0^{[i+1,i]}$, giving the equality for $f=1$.
    
    Consider now $f=2$.
    We need to prove that $U^iU^{i+1}\cap U^{i+1}U^{i} \subseteq   U^i \cup U^{i+1}$.
    Consider $x\in U^iU^{i+1}\cap U^{i+1}U^i$.
    Hence, $x$ can be written as both a word $e_i(a)e_{i+1}(b)$ and a word $e_{i+1}(c)e_{i+1}(d)$, i.e. we have that $$\mathbf{1}+f_i(a)+f_{i+1}(b)+f_{i,i+2}(ab) = \mathbf{1} + f_i(d)+f_{i+1}(c).$$
    This is only true when $a=d$, $b=c$ and $ab=0\Leftrightarrow a=0\vee b=0$.
    When $a=0$, we get that $x=e_{i+1}(b)=e_{i+1}(c)\in U^{i+1}$; when $b=0$, we get that $x=e_i(a)=e_i(d)\in U^i$. In either case, we have that $x\in U^i\cup U^{i+1}$, as wanted.

    Now, we will consider the case $f=3$. Let $x\in U^iU^{i+1}U^i\cap U^{i+1}U^{i}U^{i+1}$, i.e.
    $$x=e_i(\alpha_1)e_{i+1}(\alpha_2)e_{i}(\alpha_3) = e_{i+1}(\alpha_4)e_{i}(\alpha_5)e_{i+1}(\alpha_6)$$ with $\alpha_k\in \FF_q$ for $k=1,\ldots,6$. 
    Hence, we have 
    \begin{align*}
        e_i(\alpha_1)e_{i+1}(\alpha_2)e_{i}(\alpha_3) &= e_{i+1}(\alpha_4)e_{i}(\alpha_5)e_{i+1}(\alpha_6)\\
        \mathbf{1} + f_i(\alpha_1+\alpha_3)+f_{i+1}(\alpha_2)+f_{i,i+2}(\alpha_1\alpha_2)&=\mathbf{1}+f_i(\alpha_5)+f_{i+1}(\alpha_4+\alpha_6)+f_{i,i+2}(\alpha_5\alpha_6),
    \end{align*}
    which implies that $\alpha_1+\alpha_3=\alpha_5$, $\alpha_2=\alpha_4+\alpha_6$ and $\alpha_1\alpha_2=\alpha_5\alpha_6$.
    Note that if, for some $k$, $\alpha_k=0$, we have trivially that $x\in U^iU^{i+1}\cup U^{i+1}U^i$. So, we will suppose that for all $k=1,\ldots,6$, $\alpha_k\neq 0$.
    
    \textit{Case $q\neq 2$:} First, note that the intersection $U^iU^{i+1}U^i\cap U^{i+1}U^{i}U^{i+1}$ has non-trivial elements which can be written will all $\alpha_k\neq 0$. Indeed, we can take $x$ in this intersection where $\alpha_5=\alpha_2$, which will imply that $\alpha_1=\alpha_6$ and $\alpha_3=\alpha_4$ (all taken to be non-zero).
    Additionally, this means that the matrix of $x$ has a non-null $(i,i+2)$-entry. We remind the reader that, for $\alpha,\beta\in\FF_q$, $e_{i}(\alpha)e_{i+1}(\beta) = \mathbf{1} + f_i(\alpha)+f_{i+1}(\beta)+f_{i,i+2}(\alpha\beta)$, while $e_{i+1}(\beta)e_{i+1}(\alpha) = \mathbf{1} + f_i(\alpha)+f_{i+1}(\beta)$. Since
    any element of $U^{i+1}U^i$ is always a matrix with a null $(i,i+2)$-entry, $x\notin U^{i+1}U^i$. Moreover, given an element from $U^{i}U^{i+1}$, its $(i,i+2)$-entry is the product of its $(i,i+1)$ and $(i+1,i+2)$ entries. However, that would mean $(\alpha_1+\alpha_3)\alpha_2 = \alpha_1\alpha_2 \Leftrightarrow \alpha_2 = 0 \vee \alpha_3 = 0$, a contradiction. Therefore, $x\notin U^{i}U^{i+1}$.
    This proves that, for $q\neq 2$, $f=3$ is the minimal positive integer such that 
    $U^{[i,i+1]}_f\cap U^{[i+1,i]}_f\neq U^{[i,i+1]}_{f-1}\cup U^{[i+1,i]}_{f-1}$.
    
    \textit{Case $q=2$:} As we are forcing that all $\alpha_k\neq 0$, with $q=2$ there is only one option, i.e. $\alpha_k=1$ for all $k$. However, this means that $\alpha_5 = 1 + 1 = 0$, a contradiction. Indeed, at least one of ${\alpha_1,\alpha_2,\alpha_3}$ and at least one of ${\alpha_4,\alpha_5,\alpha_6}$ must be zero.
    Therefore, all elements in the intersection $U^iU^{i+1}U^i\cap U^{i+1}U^{i}U^{i+1}$ can be written in either $U^iU^{i+1}$ or $U^{i+1}U^{i}$.
    Therefore, contrary to the case $q\neq 2$, whenever $q=2$ we have $ U^iU^{i+1}U^i\cap U^{i+1}U^{i}U^{i+1}= U^iU^{i+1}\cup U^{i+1}U^i$, i.e. $f=3$ is \emph{not} the minimal positive integer such that $U^{[i,i+1]}_f\cap U^{[i+1,i]}_f\neq U^{[i,i+1]}_{f-1}\cup U^{[i+1,i]}_{f-1}$. 

    Finally, we will prove that, whenever $q=2$, $f=4$ is the minimal positive integer satisfying $U^{[i,i+1]}_f\cap U^{[i+1,i]}_f\neq U^{[i,i+1]}_{f-1}\cup U^{[i+1,i]}_{f-1}$. Take $x\in U^iU^{i+1}U^iU^{i+1}\cap U^{i+1}U^{i}U^{i+1}U^i$  such that $$x=(e_i(1)e_{i+1}(1))^2 = (e_{i+1}(1)e_i(1))^2=e_{i,i+2}(1).$$
    An element of $U^iU^{i+1}U^i$ is always written as $$e_i(\alpha_1)e_{i+1}(\alpha_2)e_{i}(\alpha_3)=\mathbf{1}+f_i(\alpha_1+\alpha_3)+f_{i+1}(\alpha_2)+f_{i,i+2}(\alpha_1\alpha_2).$$
    For $x$ to be in $U^iU^{i+1}U^i$, we would need $\alpha_1\alpha_2=1$ and $\alpha_2=0$, a contradiction. The same argument can be said for elements of $U^{i+1}U^iU^{i+1}$. Hence, $x=e_{i,i+2}(1)\notin U^iU^{i+1}U^i\cup U^{i+1}U^iU^{i+1}$. Therefore, for $q\neq 2$, $f=4$ is the minimal positive integer such that 
    $U^{[i,i+1]}_f\cap U^{[i+1,i]}_f\neq U^{[i,i+1]}_{f-1}\cup U^{[i+1,i]}_{f-1}$.
\end{proof}

\subsection{The Geometry of Unitriangular Groups}

\begin{theorem}\label{thm:CG_UTnq}
Let $U:=UT(n,\FF_q)$ be the upper unitriangular group of size $n$ over the field $\FF_q$.

    Consider the coset incidence system $\UU\TT(n,q):=(U, (U_i)_{i\in I})$, where $I=\{1,\ldots,n-1\}$ and \linebreak $U_i := \langle e_{j}(t) \mid t\in T, j\in I\setminus\{i\} \rangle$.

    \begin{enumerate}
        \item $U$ acts flag-transitively on $\UU\TT(n,q)$;
        \item $\UU\TT(n,q)$ is residually connected;
        \item $\UU\TT(n,q)$ is thick for $q\neq 2$; $\UU\TT(n,q)$ is thin for $q=2$;
        \item $\UU\TT(n,q)$ has linear Buekenhout diagram,
        $$\xymatrix@-1pc{*{\circ} \ar@{-}[rr]^{4\ 3\ 4}_(0.01){q-1}  && *{\circ}\ar@{-}[rr]^{4\ 3\ 4}_(0.01){q-1}  && *{\circ}\ar@{..}[rr]_(0.01){q-1}  && *{\circ}\ar@{-}[rr]^{4\ 3\ 4}_(0.01){q-1} && *{\circ}\ar@{-}[rr]^{4\ 3\ 4}_(0.01){q-1}_(0.99){q-1}  && *{\circ}\\}$$
        when $q\neq 2$, otherwise
        $$\xymatrix@-1pc{*{\circ} \ar@{-}[rr]^{4}_(0.01){1}  && *{\circ}\ar@{-}[rr]^{4}_(0.01){1}  && *{\circ}\ar@{..}[rr]_(0.01){1}  && *{\circ}\ar@{-}[rr]^{4}_(0.01){1} && *{\circ}\ar@{-}[rr]^{4}_(0.01){1}_(0.99){1}  && *{\circ}\\}$$
    \end{enumerate}
\end{theorem}
\begin{proof}
    Point (1) and (2) are a direct consequence of Theorem~\ref{thm:CG_of_Pattern}. By Proposition~\ref{prop:Pattern_B=1equiv} and Proposition~\ref{prop:UniT_IP}, we have that $B=\langle \mathbf{1}\rangle$. By Theorem~\ref{thm:CG_of_Pattern}(4)--(5), we have point (3) of the statement of this theorem.

    Finally, point (4) is a consequence of Propositions~\ref{prop:UTnq_lin_diag} and ~\ref{prop:diag_props} and the fact that, as $\UU\TT(n,q)$ is a residually connected flag-transitive geometry, its rank $1$ residues are isomorphic to $U^i=\langle e_{i}(t)\mid t\in T\rangle\cong (C_p)^m$. Hence, each vertex of the Buekenhout diagram is labeled by $o_i=|U^i|-1=q-1$.
    
\end{proof}

By Theorem~\ref{thm:CG_UTnq} and Corollary~\ref{coro:Pattern_with_q=2_hyper}, we have that $\UU\TT(n,2)$ is a regular hypertope with a linear Buekenhout diagram, which is the exact definition of an abstract regular polytope (see~\cite{hypertopes}). Indeed, for each $n\geq 3$, we have an abstract regular polytope $\UU\TT(n,2)$.

\begin{coro}
    $\UU\TT(n,2)$ is an abstract regular polytope of type $\{4^{n-2}\}$ with automorphism group the quotient of the Coxeter group $[4^{n-2}]=\langle s_1,\ldots,s_{n-1}\rangle$ by the extra relations
    $(s_is_{i+1}s_{i+2})^4=1$ for $i\in\{1,\ldots,n-3\}$.
\end{coro}
\begin{proof}
    The above follows from Theorem~\ref{thm:CG_UTnq}, Corollary~\ref{coro:Pattern_with_q=2_hyper}, and the definition of an abstract regular polytope.
    Regarding the group relations, this follows directly from the group relations of $UT(n,\FF_q)$ determined in~\cite[Theorem 1]{biss2001presentation}.
\end{proof}

The relations above define the automorphism group of an abstract regular polytope where the subgroup generated by triples of consecutive generators is isomorphic to the the automorphism group of a toroidal regular map $\{4,4\}_{(2,2)}$.
There has been a lot of work around very similar polytopes. Indeed, Marston Conder has described a similar family of abstract regular polytope of type $\{4^{n-2}\}$, where each toroidal residue is instead $\{4,4\}_{(2,0)}$~\cite{Conder2013}. Additionally, in~\cite{Cunningham2026}, the authors have described a flat regular polytope of type $\{4^{n-2}\}$, with toroidal residue $\{4,4\}_{(2,2)}$. However, due to the flatness condition, this polytope is smaller than our geometry $\UU\TT(n,2)$. Indeed, in~\cite{Cunningham2026}, it is mentioned that $\UU\TT(n,2)$ is a $2$-cover of the flat polytope described.
In correspondence with Marston Conder and Dong-Dong Hou, we were made aware that this family was either known or expected to exist.
Particularly, some members of $\UU\TT(n,2)$ are present in Marston Conder's polytope atlas~\cite{ConderAtlas} for $n=4,5$. Nevertheless, the full description of this family, up to our knowledge, was yet to be published.

\subsection{Automorphisms of $\UU\TT(n,q)$}

Even though we have that $UT(n,\FF_q)$ acts flag-transitively on $\UU\TT(n,q)$, this does not necessarily describe all possible automorphisms of $\UU\TT(n,q)$. Indeed, as we will see, this is only the case whenever $q=2$. 
In this section, we will characterize some of the extra automorphisms of $\UU\TT(n,q)$.

\subsubsection{Regular Hypertope case}
For $q=2$, we have that $\UU\TT(n,q)$ is a regular hypertope.
It is a well-known result for regular hypertopes built as a coset geometry $\Gamma = \Gamma(G,(G_i)_{i\in I})$, with trivial Borel subgroup, that $G = \Aut(\Gamma)$ (see~\cite{hypertopes} for more details). Hence, we have that $\Aut(\UU\TT(n,2)) = UT(n,\FF_2)$.

\subsubsection{Case $n=2$}
We will now consider the case $n=2$. We have not discussed much of this case before as it is quite particular.
For $n=2$, we have exactly $UT(2,\FF_q) = \{e_{1}(\alpha) \mid \alpha\in \FF_q\}\cong (\FF_q,+) $, having has many elements as $\FF_q$.
In this group, there is only one maximal parabolic subgroup, which is the trivial one, hence $\UU\TT(2,q) =(UT(2,\FF_q), (\{\textbf{1}\}))$ is a rank 1 coset geometry. Interestingly, in this particular case, the minimal parabolic is considered to be the whole group $UT(2,\FF_q)$.
Note that the only flags possible are the exactly the elements of $UT(2,\FF_q)$, which are the cosets of $\{\textbf{1}\}$ in this group, and  $UT(2,\FF_q)$ acts transitively trivially on this set. 

Geometrically, the incidence graph of $\UU\TT(2,q)$ is a graph with $q$ vertices, all disconnected from each other (there exists only one type of elements and, hence, there is only self-incidence).
In this case, $Aut(\UU\TT(2,q)) \cong  \Sym(q)$, the symmetric group on $q$ points. This case is a particular example where extra automorphisms of a coset geometry $(G,(G_i)_{i\in I}))$ do not come from the automorphism group of $G$. We point out that this does not contradict the regular hypertope case, since when $q=2$, we have $UT(2,\FF_2) \cong Sym(2)$.

\subsubsection{Case $n>2$}

Lastly, we will consider $n>2$.

\begin{proposition}
    Let $\phi\in \Aut(UT(n,\FF_q))$. Then, $\phi$ is an automorphism of $\UU\TT(n,q)$ if and only if $\phi\in \langle \Aut_D,\Aut_F\rangle$, where
    $\Aut_D$ is the diagonal automorphisms of $UT(n,\FF_q)$, and
$\Aut_F$ is the field automorphisms of $\FF_q$.
\end{proposition}
\begin{proof}
$(\Leftarrow)$
    We will first prove that, if $\phi\in \Aut_D$ or $\phi \in \Aut_F$, then $\phi$ is an automorphism of $\UU\TT(n,q)$.
    
    Suppose $\phi$ is a diagonal automorphism of $UT(n,\FF_q)$, i.e. this automorphism acts by conjugation of a diagonal matrix $diag(\lambda_1,\lambda_2,\ldots,\lambda_n)$, with $\lambda_i$ non-zero terms of $\FF_q$. Let $D = diag(\lambda_1,\lambda_2,\ldots,\lambda_n)$ be the diagonal matrix associated with $\phi$ such that it send the generators of $UT(n,\FF_q)$, $e_i(t)$, to $\phi(e_i(t)) = D e_i(t) D^{-1}$, with $i\in\{1,\ldots, n-1\}$ and $t\in T$. We can simplify and write $\phi(e_{i}(t)) = e_{i}(\lambda_i t \lambda_{i+1}^{-1})$.
    Notice that the kernel of the action by diagonal matrices is the case where $\lambda_1=\lambda_2=\ldots=\lambda_n$. Hence, $Aut_D$ does not contain those cases.
    For each $i\in \{1,\ldots,n-1\}$,  $\phi(gU_{i}) = \phi(g)U_{i}$, fixing the subgroup $U_{i}$. Indeed, for all generators of $U_{i}$, i.e. $e_{k}(t)$ for $k\in \{1,\ldots,n-1\}\setminus\{i\}$ and $t\in T$,
   $\phi(e_{k}(t))= e_{k}(\lambda_k t\lambda_{k+1}^{-1})\in U_{i}$.
    Therefore, this action is an automorphism of $\UU\TT(n,q)$ which is type-preserving.

Suppose $\phi$ a field automorphism of $\FF_q$. Then we have that the action of $\phi$ on the set of generators of $UT(n,\FF_q)$ is  $\phi( e_{i}(t)) = e_{i}(t^{\mu_{\phi}})$, where $\mu_\phi : \FF_q\to \FF_q$ is a field automorphism.
    First, note that for all $i\in \{1,\ldots,n-1\}$, $\phi(U_{i})=U_{i}$. Indeed, given a generator of $U_{i}$, i.e. $e_{k}(t)$ for $k\in \{1,\ldots,n-1\}\setminus\{i\}$ and $t\in T$, we have
   $\phi(e_{k}(t)) = e_{k}(t^{\mu_\phi})\in U_{i}$.
    Additionally, any coset $gU_{i}$ is mapped under $\phi$ to $\phi(g)U_{i}$. Therefore, this action is an automorphism of $\UU\TT(n,q)$ which is type-preserving.

    Since the composition of automorphisms is still an automorphism, we have that any element of $\langle \Aut_D,\Aut_F\rangle$ will be an automorphism of $\UU\TT(n,q)$.

$(\Rightarrow)$  We will now to prove that any automorphism of $\UU\TT(n,q)$ coming from $\Aut(UT(n,\FF_q))$ must be a product of a field automorphism and a diagonal automorphism. Let $\phi\in \Aut(\UU\TT(n,q))$. First note that since $\phi$ is an automorphism of $\UU\TT(n,q)$, it must preserve types. In particular, for each $i\in I$, $\phi$ must fix minimal parabolics $\langle e_i(t)\mid t\in T\rangle$.

Since $\phi$ preserves then minimal parabolics, then for each $i$ there exists an additive morphisms $f_{i,j}:\FF_q\to \FF_q$ such that 
$$
\phi(e_i(\alpha))=e_i(f_{i,i+1}(\alpha))
$$
for some $\alpha\in \FF_q$.
In general, since $\phi(e_{i,j}(\alpha))=e_{i,j}(\beta)$, for $\alpha,\beta\in \FF_q$, we can define the $f_{i,j}$ as:
$$
\phi(e_{i,j}(\alpha))=e_{i,j}(f_{i,j}(\alpha)).
$$

Using the commutator relations in Proposition \ref{prop:gens_props}(3), we have
\begin{equation}\label{eq:com}
  f_{ik}(\alpha \beta)=f_{i,j}(\alpha)f_{j,k}(\beta)  
\end{equation}
for all $i<j<k$ and $\alpha,\beta\in\FF_q$.

We will now prove that all morphisms $f_{i,j}$ can be written in terms of the morphism $f_{1,2}$.
We will prove by induction. Note that $f_{i,j}(\alpha) = 0$ if and only if $\alpha = 0$, otherwise, $\phi$ would not be an automorphism of $UT(n,\FF_q)$.
From Equation \ref{eq:com}, we can derive:
$$
f_{1,2}(\alpha)f_{2,3}(1)=f_{1,2}(1)f_{2,3}(\alpha) \Leftrightarrow f_{2,3}=\frac{f_{2,3}(1)}{f_{1,2}(1)}f_{1,2},
$$
This is well defined as $f_{1,2}(1)\neq 0$. Taking this as our base case, we can use induction to prove that each morphism $f_{i,i+1}$ can be written as $f_{i,i+1} = c_{i,i+1} f_{1,2}$, for some constant $c_{i,i+1}\in \FF_q^\times$. 
Now, using each of the $f_{i,i+1}$, we will use induction on $j-i$ to prove the remaining cases, with $j>i+1$. Using Equation \ref{eq:com} again, we get
$$
f_{i,j}(\alpha \beta)=f_{i,j-1}(\alpha)f_{j-1,j}(\beta),
$$
which by setting $\beta=1$ we obtain
$$
f_{i,j}(\alpha)=f_{i,j-1}(\alpha)f_{j-1,j}(1)
$$
From the equation above, and using induction on $j-i$, one can conclude that $f_{i,j}=c_{i,j}f_{1,2}$ for a constanct $c_{i,j}\in \FF_q^\times$.

From now on, we will consider  $f_{i,j}=c_{i,j}f_{1,2}$. Let us denote $g := f_{1,2}$. Substituting $f_{i,j}=c_{i,j}f_{1,2} = c_{i,j}g$ into Equation \ref{eq:com}, we get
\begin{equation}\label{eq:constant}
c_{i,k}g(\alpha\beta)=c_{i,j}c_{j,k}g(\alpha)g(\beta),
\end{equation}
which, by setting both $\alpha$ and $\beta$ to $1$, gives
$$
c_{i,k}g(1)=c_{i,j}c_{j,k}g(1)g(1).
$$
As $g(1)\neq 0$ and every $c_{i,k}\neq 0$, we can write
$$
g(1)^{-1}=\frac{c_{i,j}c_{j,k}}{c_{i,k}}
$$

Consider $\sigma = g(1)^{-1} g$. Then, for all $i<j<k$, we have
$$f_{i,j}=c_{i,j}f_{1,2}= c_{i,j} g(1) g(1)^{-1}g=c_{i,j}\cdot \frac{c_{i,k}}{c_{i,j}c_{j,k}}\sigma=\frac{c_{i,k}}{c_{j,k}}\sigma$$
In particular, taking $k=n$, we have 
$$f_{i,j} = \frac{c_{i,n}}{c_{j,n}}\sigma \Rightarrow \phi(e_{i,j}(\alpha)) = e_{i,j}\left(\frac{c_{i,n}}{c_{j,n}}\sigma(\alpha)\right).$$
First, let us prove that $\sigma$ is a field automorphism of $\FF_q$.
As $f_{1,2}$ is additive, then $\sigma$ is also also additive. Moreover, by Equation~\ref{eq:constant} we have
\begin{align*}
g(\alpha\beta)&= \frac{c_{i,j}c_{j,k}}{c_{i,k}}g(\alpha)g(\beta) \Leftrightarrow\\
    \frac{c_{i,j}c_{j,k}}{c_{i,k}}g(\alpha\beta)&= \frac{c_{i,j}c_{j,k}}{c_{i,k}}g(\alpha)\frac{c_{i,j}c_{j,k}}{c_{i,k}}g(\beta) \Leftrightarrow\\
\sigma(\alpha\beta) &= \sigma(\alpha)\sigma(\beta)
\end{align*}
Then $\sigma$ is a multiplicative morphism. From the fact that $\phi$ is an automorphism, one can see that $\sigma$ is a bijection, and we also have $\sigma(1)=1$. Then $\sigma$ is a field automorphism of $\FF_q$. 
Finally, by defining the diagonal matrix $h$ with entries $h_{i,i}:=c_{i,n}$, then $\phi$ is equal to the automorphism $C_h\circ\bar\sigma$, where $\bar\sigma$ is the automorphism induced by the field automorphism $\sigma$ on the group, and $C_h$ is the conjugacy by the diagonal matrix $h$.

\end{proof}

Following the above results, we conjecture that, for $n\geq 2$, all automorphisms of $\UU\TT(n,q)$ are either in $UT(n,\FF_q)$, $\Aut_D(UT(n,\FF_q))$, $\Aut_F(UT(n,\FF_q))$, or are any composition of the above.

\begin{conj} For $n\geq 2$,
    $\Aut(\UU\TT(n,q)) \cong UT(n,\FF_q)\rtimes \langle \Aut_D,\Aut_F\rangle$.
\end{conj}

This has been tested to be true for small values of $q$ and $n$ using computational tools.
Additionally, whenever $q=2$, both $\Aut_D$ and $\Aut_F$ are trivial, giving the results obtain previously for the regular hypertopes.

\section{From $\UU\TT_{\widehat{I}}(n,p^m)$ to $\UU\TT_{\widehat{I}}(n,p)$}

In this last section, we will explore the relationship between the geometric structures described in this paper, whenever we vary solely the value of $m$. In particular, we will draw a connection between $\UU\TT_{\widehat{I}}(n,p^m)$ and $\UU\TT_{\widehat{I}}(n,p)$.
To do this, consider the incidence system defined below, which is a special case of the \emph{absolute geometry} (see~\cite{buekenhout2013diagram}).

\begin{definition}
    Let $\Gamma :=(\cup_{i\in I}X_i,I,t,\star)$ be an incidence system and let $\phi \in Aut(\Gamma)$ be a type-preserving automorphism of $\Gamma$. Moreover, for every $i\in I$, suppose that $X_i^\phi \neq \emptyset$, where $X_i^\phi$ is the set of fixed points of $X_i$ by $\phi$.
    The \emph{fixed-point incidence system of $\phi$} is the incidence system $\Gamma^\phi :=(\cup_{i\in I^\phi}X_i^\phi, I^\phi, t^\phi,\star^\phi)$ such that 
    \begin{itemize}
        \item $X_i^\phi$ is the set of fixed points of $X_i$ by $\phi$;
        \item $I^\phi = I$;
        \item $t^\phi(X_i^\phi) = i$;
        \item for $x,y \in \cup_{i\in I^\phi} X_i^\phi$, $ x \star^\phi y$ if and only if $x\star y$.
    \end{itemize}
\end{definition}

Note that we will always consider $m>1$ such that $q=p^m $ is always different from a prime. Additionally, given a set $S$ and a field automorphism $F$, we will consider the $S^F$ to be the set of fixed points of $S$ by $F$, that is $\{s \mid s\in S, F(s)=s\}$. This notation will be used in the proof of the following theorem in different sets, with some abuse of notation.

\begin{theorem}
    Consider a closed set $\widehat{I}\subseteq \Delta_n$, with covering relations $I$, and let $q$ be a power of a prime $p$, $q=p^m$.
    Let $\Gamma = \UU\TT_{\widehat{I}}(n,q)$ and $F$ be a field automorphism of $\FF_q$. Then $\Gamma^F \cong \UU\TT_{\widehat{I}}(n,p)$.
\end{theorem}
\begin{proof}
    Suppose that $F$ is the field automorphism $F(x)=x^p$, this is, the Frobenius automorphism. The case where $F(x)=x^{p^k}$, with $k>1$, is similar.
    For $(i,j)\in I$, we will define $U_{(i,j)}(q) := \langle e_{k,l}(\alpha) \mid (k,l)\in I\setminus\{(i,j)\}, \alpha\in \FF_q\rangle$.
    Consider $X_n(q)=\{gU_{i,j}(q)\mid (i,j)\in I,\;g\in UT_{\widehat{I}}(\FF_q)\}$ and $X_n(p)=\{gU_{i,j}(p)\mid (i,j)\in I,\;g\in UT_{\widehat{I}}(\FF_p)\}$. 
    
    Since $F$ is a group automorphism that fixes the subgroups $U_{i,j}(q)$ (or, equivalently, it acts as an automorphism of $\UU\TT(n,q)$), then $F$ naturally defines a function on the elements of $X_n(q)$.
    Recall that $(UT_{\widehat{I}}(n,\FF_q))^F=UT_{\widehat{I}}(n,\FF_p)$ and $(U_{(i,j)}(q))^F=U_{(i,j)}(p)$. Considering $\Gamma = \UU\TT_{\widehat{I}}(n,q)$, the set of elements of $\Gamma^{F}$ is $X_n(q)^F$. Since $F$ acts as an automorphism of $\UU\TT(n,q)$, $I^F=I$ and the type function follows from the previous one. To prove that $\Gamma^F \cong \UU\TT_{\widehat{I}}(n,p)$, we then need only to show these have the same elements and the same incidences.

    First, we will describe properly the elements of $X_n(q)^F$. Consider an element of $X_n(q)^F$, i.e. a coset of $X_n(q)$, $gU_{(i,j)}(q)$, which is fixed by $F$. In other words, for elements of $X_n(q)^F$, we have $F(gU_{(i,j)}(q))=gU_{(i,j)}(q)$. For this elements, we have $F(gU_{(i,j)}(q)) = F(g)U_{(i,j)}(q)=gU_{(i,j)}(q)$, which is equivalent to $g^{-1}F(g)\in U_{(i,j)}(q)$.  By Lang's Theorem (see \cite{Lang1956}) there exists an $h\in  U_{(i,j)}(q)$ such that $h^{-1}F(h)=g^{-1}F(g)$. For $k=gh^{-1}$, we can show that $k$ is $F$-invariant. Indeed,
    \begin{align*}
        F(k)&=F(g)F(h)^{-1}\\
            &=[gh^{-1}F(h)]F(h)^{-1}\\
            &=gh^{-1}=k.
    \end{align*}
    Then we have that $k\in (UT_{\widehat{I}}(n,\FF_q))^F=UT_{\widehat{I}}(n,\FF_p)$. Hence, we can rewrite our fixed points by taking transversals in $UT_{\widehat{I}}(n,\FF_p)$ since $kU_{i,j}(q)=gh^{-1}U_{i,j}(q)=gU_{i,j}(q)$. Hence, we have that $X_n(q)^F=\{gU_{(i,j)}(q)\mid (i,j)\in I,\;g\in UT_{\widehat{I}}(n,p)\}$.

    Now, let us prove that $X_n(q)^F$ is in bijection with $X_n(p)$, with elements of the same type being mapped to each other. Indeed, we will see that each fixed point $kU_{(i,j)}(q)$ of $X_n(q)^F$ can be mapped to uniquely to the coset $kU_{(i,j)}(p)$, with $k\in UT_{\widehat{I}}(n,p)$.
    We remind the reader that $(kU_{i,j}(q))^F$ is the set of elements of the coset $kU_{i,j}(q)$ which are fixed by $F$.
    It is not hard to check that, given $kU_{i,j}(q)\in X_n(q)^F$, with $k\in UT_{\widehat{I}}(n,p)$, we have $kU_{i,j}(p) = kU_{i,j}(q)^F\subseteq (kU_{i,j}(q))^F$, since $U_{(i,j)}(p)\subseteq U_{(i,j)}(q)$. For the other inclusion, given an $x\in (kU_{i,j}(q))^F$, then there exists a $y\in U_{(i,j)}(q)$ such that $x=ky$. As $x$ is inside of a fixed point of $F$, we have $F(x)=x$ and, by the previous construction of $k$, $F(k)=k$.
    Then,
    $$
    ky=x=F(x)=F(ky)=kF(y).
    $$
    We can then conclude that $F(y)=y$ and so $(kU_{i,j}(q))^F=kU_{i,j}(q)^F=kU_{i,j}(p)$.
    Hence, the elements of $X_n(q)^F=\{gU_{(i,j)}(q)\mid (i,j)\in I,\;g\in UT_{\widehat{I}}(n,p)\}$, are in bijection with $X_n(p)$. 
    
    Finally, let us prove now that the incident relation between elements of $X_n(q)^F$ is equivalent to the one between elements in $X_n(p)$.
    Given $g_1U_{(i,j)}(p) \cap g_2U_{(s,t)}(p) \neq \emptyset$, it is direct that $g_1U_{(i,j)}(q) \cap g_2U_{(s,t)}(q) \neq \emptyset$. Let us now prove the reverse.
    Consider two elements of $X_n(q)^F$, $g_1U_{(i,j)}(q)$ and $g_2U_{(s,t)}(q)$, such that 
    $g_1U_{(i,j)}(q)\cap g_2U_{(s,t)}(q)\neq \emptyset$. From what we have proved before, we can take $g_1,g_2\in UT_{\widehat{I}}(n,\FF_q)^F = UT_{\widehat{I}}(n,\FF_p)$. Note that since $UT_{\widehat{I}}(n,q)$ is always incidence transitive in $\UU\TT_{\widehat{I}}(n,q)$ (see~\cite[Lemma 1.8.6]{buekenhout2013diagram}), there exists a $h\in UT_{\widehat{I}}(n,\FF_q)$ such that $g_1U_{(i,j)}(q)=hU_{(i,j)}(q)$ and $g_2U_{(s,t)}(q)=hU_{(s,t)}(q)$ such that $$h(U_{(i,j)}(q)\cap U_{(s,t)}(q))=g_1U_{(i,j)}(q)\cap g_2U_{(s,t)}(q).$$ Since $F(h(U_{(i,j)}(q)\cap U_{(s,t)}(q)))=h(U_{(i,j)}(q)\cap U_{(s,t)}(q))$, then $ (g_1U_{(i,j)}(q)\cap g_2U_{(s,t)}(q))^F\neq\emptyset$.
    Since the fixed points of the intersection is equal to the intersection of the fixed points, we have
    $$
    \emptyset\neq (g_1U_{(i,j)}(q)\cap g_2U_{(s,t)}(q))^F= (g_1U_{(i,j)}(q))^F\cap (g_2U_{(s,t)}(q))^F=g_1(U_{(i,j)}(q))^F\cap g_2(U_{(s,t)}(q))^F.
    $$
    Since $g_1(U_{(i,j)}(q))^F=g_1(U_{(i,j)}(p))$ and $g_2(U_{(s,t)}(q))^F = g_2(U_{(s,t)}(p))$, we get that $g_1U_{(i,j)}(p) \cap g_2U_{(s,t)}(p) \neq \emptyset$, as we wanted.
    Then the incident relations between elements of $X_n(q)^F$ is equivalent to the one between elements of $X_n(p)$. And so we have that $\Gamma^F \cong \UU\TT_{\widehat{I}}(n,p)$.
\end{proof}

\begin{exmp}

Consider the unitriangular group of order $3$ over $\FF_4$, $UT(3,\FF_4)$, and let $1,a\in \FF_4$ be its generators as an additive group of the field. By considering the following maximal parabolic subgroups,
     $H_1 = \langle e_{1,2}(t)| t \in T\rangle$
and $ H_2 = \langle e_{2,3}(t)| t \in T\rangle$, we get the incident graph given in Figure~\ref{fig:IncGraphAbs}.

\begin{figure}
\begin{tikzpicture}[>=latex',line join=bevel,scale=0.15]

\node[scale=0.4] (1) at (3192.3bp,1271.3bp) [draw,fill=gray,ellipse,inner xsep=2pt,inner ysep=2pt] {$\begin{bmatrix}1&0&0\\0&1&0\\0&0&1\end{bmatrix}\,H_1$};
  \node[scale=0.4] (17) at (3103.7bp,979.19bp) [draw,fill=gray,ellipse,inner xsep=2pt,inner ysep=2pt] {$\begin{bmatrix}1&0&0\\0&1&0\\0&0&1\end{bmatrix}\,H_2$};
  \node[scale=0.4] (18) at (1665.2bp,18.0bp) [draw,fill=gray,ellipse,inner xsep=2pt,inner ysep=2pt] {$\begin{bmatrix}1&1&0\\0&1&0\\0&0&1\end{bmatrix}\,H_2$};
  \node[scale=0.4] (25) at (2959.8bp,2440.1bp) [draw,fill=white,ellipse,inner xsep=2pt,inner ysep=2pt] {$\begin{bmatrix}1&a&0\\0&1&0\\0&0&1\end{bmatrix}\,H_2$};
  \node[scale=0.4] (26) at (108.16bp,1575.0bp) [draw,fill=white,ellipse,inner xsep=2pt,inner ysep=2pt] {$\begin{bmatrix}1&a^2&0\\0&1&0\\0&0&1\end{bmatrix}\,H_2$};
  \node[scale=0.4] (2) at (1361.4bp,3102.2bp) [draw,fill=gray,ellipse,inner xsep=2pt,inner ysep=2pt] {$\begin{bmatrix}1&0&1\\0&1&0\\0&0&1\end{bmatrix}\,H_1$};
  \node[scale=0.4] (19) at (1665.2bp,3132.1bp) [draw,fill=gray,ellipse,inner xsep=2pt,inner ysep=2pt] {$\begin{bmatrix}1&1&1\\0&1&0\\0&0&1\end{bmatrix}\,H_2$};
  \node[scale=0.4] (20) at (1069.4bp,3013.6bp) [draw,fill=gray,ellipse,inner xsep=2pt,inner ysep=2pt] {$\begin{bmatrix}1&0&1\\0&1&0\\0&0&1\end{bmatrix}\,H_2$};
  \node[scale=0.4] (27) at (3192.3bp,1878.8bp) [draw,fill=white,ellipse,inner xsep=2pt,inner ysep=2pt] {$\begin{bmatrix}1&a^2&1\\0&1&0\\0&0&1\end{bmatrix}\,H_2$};
  \node[scale=0.4] (28) at (370.57bp,710.0bp) [draw,fill=white,ellipse,inner xsep=2pt,inner ysep=2pt] {$\begin{bmatrix}1&a&1\\0&1&0\\0&0&1\end{bmatrix}\,H_2$};
  \node[scale=0.4] (3) at (2530.3bp,280.41bp) [draw,fill=white,ellipse,inner xsep=2pt,inner ysep=2pt] {$\begin{bmatrix}1&0&a\\0&1&0\\0&0&1\end{bmatrix}\,H_1$};
  \node[scale=0.4] (21) at (2766.2bp,474.05bp) [draw,fill=white,ellipse,inner xsep=2pt,inner ysep=2pt] {$\begin{bmatrix}1&1&a\\0&1&0\\0&0&1\end{bmatrix}\,H_2$};
  \node[scale=0.4] (22) at (2261.1bp,136.52bp) [draw,fill=white,ellipse,inner xsep=2pt,inner ysep=2pt] {$\begin{bmatrix}1&0&a\\0&1&0\\0&0&1\end{bmatrix}\,H_2$};
  \node[scale=0.4] (29) at (800.16bp,280.41bp) [draw,fill=white,ellipse,inner xsep=2pt,inner ysep=2pt] {$\begin{bmatrix}1&a^2&a\\0&1&0\\0&0&1\end{bmatrix}\,H_2$};
  \node[scale=0.4] (30) at (1969.0bp,3102.2bp) [draw,fill=white,ellipse,inner xsep=2pt,inner ysep=2pt] {$\begin{bmatrix}1&a&a\\0&1&0\\0&0&1\end{bmatrix}\,H_2$};
  \node[scale=0.4] (4) at (370.57bp,2440.1bp) [draw,fill=white,ellipse,inner xsep=2pt,inner ysep=2pt] {$\begin{bmatrix}1&0&a^2\\0&1&0\\0&0&1\end{bmatrix}\,H_1$};
  \node[scale=0.4] (23) at (226.69bp,2170.9bp) [draw,fill=white,ellipse,inner xsep=2pt,inner ysep=2pt] {$\begin{bmatrix}1&0&a^2\\0&1&0\\0&0&1\end{bmatrix}\,H_2$};
  \node[scale=0.4] (24) at (564.21bp,2676.0bp) [draw,fill=white,ellipse,inner xsep=2pt,inner ysep=2pt] {$\begin{bmatrix}1&1&a^2\\0&1&0\\0&0&1\end{bmatrix}\,H_2$};
  \node[scale=0.4] (31) at (1069.4bp,136.52bp) [draw,fill=white,ellipse,inner xsep=2pt,inner ysep=2pt] {$\begin{bmatrix}1&a&a^2\\0&1&0\\0&0&1\end{bmatrix}\,H_2$};
  \node[scale=0.4] (32) at (2530.3bp,2869.7bp) [draw,fill=white,ellipse,inner xsep=2pt,inner ysep=2pt] {$\begin{bmatrix}1&a^2&a^2\\0&1&0\\0&0&1\end{bmatrix}\,H_2$};
  \node[scale=0.4] (5) at (2959.8bp,710.0bp) [draw,fill=white,ellipse,inner xsep=2pt,inner ysep=2pt] {$\begin{bmatrix}1&0&0\\0&1&a\\0&0&1\end{bmatrix}\,H_1$};
  \node[scale=0.4] (6) at (800.16bp,2869.7bp) [draw,fill=white,ellipse,inner xsep=2pt,inner ysep=2pt] {$\begin{bmatrix}1&0&1\\0&1&a\\0&0&1\end{bmatrix}\,H_1$};
  \node[scale=0.4] (7) at (1969.0bp,47.918bp) [draw,fill=white,ellipse,inner xsep=2pt,inner ysep=2pt] {$\begin{bmatrix}1&0&a\\0&1&a\\0&0&1\end{bmatrix}\,H_1$};
  \node[scale=0.4] (8) at (138.08bp,1878.8bp) [draw,fill=white,ellipse,inner xsep=2pt,inner ysep=2pt] {$\begin{bmatrix}1&0&a^2\\0&1&a\\0&0&1\end{bmatrix}\,H_1$};
  \node[scale=0.4] (9) at (2261.1bp,3013.6bp) [draw,fill=gray,ellipse,inner xsep=2pt,inner ysep=2pt] {$\begin{bmatrix}1&0&0\\0&1&1\\0&0&1\end{bmatrix}\,H_1$};
  \node[scale=0.4] (10) at (1361.4bp,47.918bp) [draw,fill=gray,ellipse,inner xsep=2pt,inner ysep=2pt] {$\begin{bmatrix}1&0&1\\0&1&1\\0&0&1\end{bmatrix}\,H_1$};
  \node[scale=0.4] (11) at (3103.7bp,2170.9bp) [draw,fill=white,ellipse,inner xsep=2pt,inner ysep=2pt] {$\begin{bmatrix}1&0&a\\0&1&1\\0&0&1\end{bmatrix}\,H_1$};
  \node[scale=0.4] (12) at (226.69bp,979.19bp) [draw,fill=white,ellipse,inner xsep=2pt,inner ysep=2pt] {$\begin{bmatrix}1&0&a^2\\0&1&1\\0&0&1\end{bmatrix}\,H_1$};
  \node[scale=0.4] (13) at (564.21bp,474.05bp) [draw,fill=white,ellipse,inner xsep=2pt,inner ysep=2pt] {$\begin{bmatrix}1&0&0\\0&1&a^2\\0&0&1\end{bmatrix}\,H_1$};
  \node[scale=0.4] (14) at (2766.2bp,2676.0bp) [draw,fill=white,ellipse,inner xsep=2pt,inner ysep=2pt] {$\begin{bmatrix}1&0&1\\0&1&a^2\\0&0&1\end{bmatrix}\,H_1$};
  \node[scale=0.4] (15) at (138.08bp,1271.3bp) [draw,fill=white,ellipse,inner xsep=2pt,inner ysep=2pt] {$\begin{bmatrix}1&0&a\\0&1&a^2\\0&0&1\end{bmatrix}\,H_1$};
  \node[scale=0.4] (16) at (3222.3bp,1575.0bp) [draw,fill=white,ellipse,inner xsep=2pt,inner ysep=2pt] {$\begin{bmatrix}1&0&a^2\\0&1&a^2\\0&0&1\end{bmatrix}\,H_1$};
  \draw [black, line width=1.5pt] (1) ..controls (3170.9bp,1200.6bp) and (3125.2bp,1049.9bp)  .. (17);
  \draw [black, line width=1.5pt] (1) ..controls (2998.0bp,1111.8bp) and (1859.5bp,177.46bp)  .. (18);
  \draw [gray,] (1) ..controls (3161.5bp,1426.4bp) and (2990.7bp,2285.1bp)  .. (25);
  \draw [gray,] (1) ..controls (2658.9bp,1323.8bp) and (660.56bp,1520.6bp)  .. (26);
  \draw [black, line width=1.5pt] (2) ..controls (1489.4bp,3114.8bp) and (1536.9bp,3119.5bp)  .. (19);
  \draw [black, line width=1.5pt] (2) ..controls (1256.4bp,3070.3bp) and (1174.4bp,3045.4bp)  .. (20);
  \draw [gray,] (2) ..controls (1595.9bp,2945.5bp) and (2958.2bp,2035.3bp)  .. (27);
  \draw [gray,] (2) ..controls (1267.4bp,2875.1bp) and (464.9bp,937.72bp)  .. (28);
  \draw [gray,] (3) ..controls (2597.0bp,335.17bp) and (2699.6bp,419.37bp)  .. (21);
  \draw [gray,] (3) ..controls (2446.0bp,235.4bp) and (2345.5bp,181.64bp)  .. (22);
  \draw [gray,] (3) ..controls (2128.8bp,280.41bp) and (1221.5bp,280.41bp)  .. (29);
  \draw [gray,] (3) ..controls (2480.8bp,529.29bp) and (2018.6bp,2852.8bp)  .. (30);
  \draw [gray,] (4) ..controls (334.76bp,2373.1bp) and (262.82bp,2238.5bp)  .. (23);
  \draw [gray,] (4) ..controls (421.49bp,2502.1bp) and (513.19bp,2613.9bp)  .. (24);
  \draw [gray,] (4) ..controls (438.06bp,2217.6bp) and (1001.8bp,359.22bp)  .. (31);
  \draw [gray,] (4) ..controls (767.43bp,2519.0bp) and (2126.6bp,2789.4bp)  .. (32);
  \draw [gray,] (5) ..controls (2995.7bp,777.0bp) and (3067.6bp,911.59bp)  .. (17);
  \draw [gray,] (5) ..controls (2908.9bp,647.96bp) and (2817.2bp,536.22bp)  .. (21);
  \draw [gray,] (5) ..controls (2990.7bp,865.15bp) and (3161.5bp,1723.8bp)  .. (27);
  \draw [gray,] (5) ..controls (2638.9bp,612.64bp) and (1394.7bp,235.21bp)  .. (31);
  \draw [gray,] (6) ..controls (884.37bp,2914.7bp) and (984.96bp,2968.5bp)  .. (20);
  \draw [gray,] (6) ..controls (733.55bp,2815.0bp) and (631.36bp,2731.2bp)  .. (24);
  \draw [gray,] (6) ..controls (712.59bp,2705.9bp) and (195.21bp,1737.9bp)  .. (26);
  \draw [gray,] (6) ..controls (1074.8bp,2924.3bp) and (1694.8bp,3047.6bp)  .. (30);
  \draw [gray,] (7) ..controls (1841.3bp,35.343bp) and (1793.7bp,30.652bp)  .. (18);
  \draw [gray,] (7) ..controls (2074.0bp,79.769bp) and (2156.1bp,104.67bp)  .. (22);
  \draw [gray,] (7) ..controls (1706.5bp,156.62bp) and (633.9bp,600.92bp)  .. (28);
  \draw [gray,] (7) ..controls (2018.5bp,296.8bp) and (2480.7bp,2620.4bp)  .. (32);
  \draw [gray,] (8) ..controls (332.45bp,2038.3bp) and (1470.9bp,2972.6bp)  .. (19);
  \draw [gray,] (8) ..controls (159.54bp,1949.5bp) and (205.24bp,2100.2bp)  .. (23);
  \draw [gray,] (8) ..controls (600.33bp,1970.8bp) and (2506.9bp,2350.0bp)  .. (25);
  \draw [gray,] (8) ..controls (213.93bp,1695.7bp) and (724.13bp,463.97bp)  .. (29);
  \draw [black, line width=1.5pt] (9) ..controls (2347.3bp,2805.3bp) and (3017.3bp,1187.9bp)  .. (17);
  \draw [black, line width=1.5pt] (9) ..controls (2078.3bp,3049.9bp) and (1848.0bp,3095.7bp)  .. (19);
  \draw [gray,] (9) ..controls (2156.1bp,3045.4bp) and (2074.0bp,3070.3bp)  .. (30);
  \draw [gray,] (9) ..controls (2345.1bp,2968.6bp) and (2445.3bp,2915.1bp)  .. (32);
  \draw [black, line width=1.5pt] (10) ..controls (1489.4bp,35.313bp) and (1536.9bp,30.642bp)  .. (18);
  \draw [black, line width=1.5pt] (10) ..controls (1336.4bp,301.93bp) and (1094.5bp,2758.3bp)  .. (20);
  \draw [gray,] (10) ..controls (1217.9bp,107.38bp) and (945.31bp,220.29bp)  .. (29);
  \draw [gray,] (10) ..controls (1256.8bp,79.67bp) and (1175.4bp,104.36bp)  .. (31);
  \draw [gray,] (11) ..controls (3017.5bp,1962.7bp) and (2347.5bp,345.27bp)  .. (22);
  \draw [gray,] (11) ..controls (2676.5bp,2255.9bp) and (1000.3bp,2589.3bp)  .. (24);
  \draw [gray,] (11) ..controls (3067.9bp,2237.9bp) and (2996.0bp,2372.5bp)  .. (25);
  \draw [gray,] (11) ..controls (3125.2bp,2100.2bp) and (3170.9bp,1949.5bp)  .. (27);
  \draw [gray,] (12) ..controls (662.95bp,892.41bp) and (2339.2bp,558.98bp)  .. (21);
  \draw [gray,] (12) ..controls (226.69bp,1135.5bp) and (226.69bp,2014.4bp)  .. (23);
  \draw [gray,] (12) ..controls (205.35bp,1086.5bp) and (129.36bp,1468.5bp)  .. (26);
  \draw [gray,] (12) ..controls (262.5bp,912.19bp) and (334.44bp,777.6bp)  .. (28);
  \draw [gray,] (13) ..controls (1000.5bp,560.83bp) and (2676.8bp,894.26bp)  .. (17);
  \draw [gray,] (13) ..controls (564.21bp,692.24bp) and (564.21bp,2458.3bp)  .. (24);
  \draw [gray,] (13) ..controls (513.29bp,536.09bp) and (421.59bp,647.83bp)  .. (28);
  \draw [gray,] (13) ..controls (630.82bp,419.38bp) and (733.01bp,335.52bp)  .. (29);
  \draw [gray,] (14) ..controls (2419.0bp,2745.1bp) and (1409.1bp,2946.0bp)  .. (20);
  \draw [gray,] (14) ..controls (2766.2bp,2457.9bp) and (2766.2bp,691.77bp)  .. (21);
  \draw [gray,] (14) ..controls (2817.1bp,2614.0bp) and (2908.8bp,2502.3bp)  .. (25);
  \draw [gray,] (14) ..controls (2699.6bp,2730.7bp) and (2597.4bp,2814.6bp)  .. (32);
  \draw [gray,] (15) ..controls (300.71bp,1469.4bp) and (1503.0bp,2934.5bp)  .. (19);
  \draw [gray,] (15) ..controls (418.35bp,1121.5bp) and (1982.6bp,285.36bp)  .. (22);
  \draw [gray,] (15) ..controls (130.99bp,1343.3bp) and (115.32bp,1502.4bp)  .. (26);
  \draw [gray,] (15) ..controls (263.16bp,1118.9bp) and (944.68bp,288.44bp)  .. (31);
  \draw [gray,] (16) ..controls (3042.6bp,1395.4bp) and (1845.1bp,197.88bp)  .. (18);
  \draw [gray,] (16) ..controls (2738.3bp,1671.3bp) and (703.47bp,2076.1bp)  .. (23);
  \draw [gray,] (16) ..controls (3215.2bp,1647.0bp) and (3199.5bp,1806.2bp)  .. (27);
  \draw [gray,] (16) ..controls (3075.6bp,1753.7bp) and (2115.5bp,2923.7bp)  .. (30);
\end{tikzpicture}
\caption{Incidence graph of $\UU\TT(3,4)$, with the vertices of the incidence subgraph of $\UU\TT(3,2)$ with darker colour.}\label{fig:IncGraphAbs}
\end{figure}
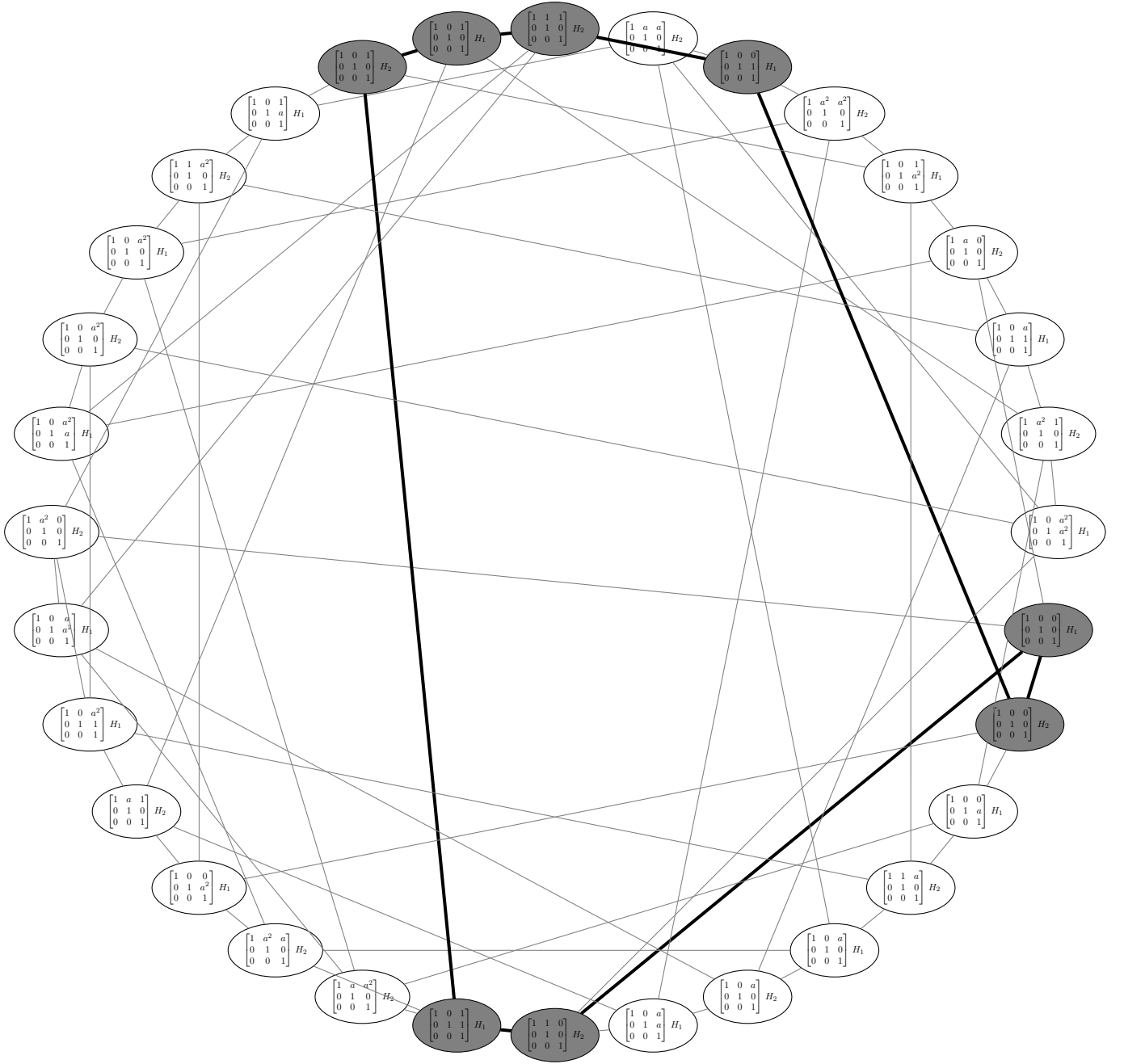

In this figure, the darker cells correspond to the elements that are fixed by the Frobenius morphism. Note that these are exactly the ones that have their transversal representative in $UT(3,\FF_2)$. One can see that the darker vertices, together with thicker edges, correspond exactly to the incident graph of $\UU\TT(3,2)$.

\end{exmp}

Taking into consideration the result above, we can go further and consider the pattern groups over the algebraic closure of the field. 
We will denote by $\overline{\FF_p}$ the algebraic closure of the finite field $\FF_p$, meaning $\overline{\FF_p} = \bigcup_{m\in \mathbb{N}} \FF_{p^m}$. 
Given a closed set $\widehat{I}\subseteq \Delta_n$ with covering relations $I$, and considering the group $U:=UT_{\widehat{I}}(n,\overline{\FF_p})$, we can define the the incident system $\UU\TT_{\widehat{I}}(n,\overline{p}):=(U,(U_{(i,j)})_{(i,j)\in I})$, with parabolic subgroups
$$U_{(i,j)}:=\langle e_{k,l}(\alpha)\mid \alpha\in \overline{\FF_p},(k,l)\in I\setminus \{(i,j\}\rangle$$
for each $(i,j)\in I$. 
Since the Frobenius morphism acts on $UT_{\widehat{I}}(n,\overline{\FF_p})$, it also acts on the elements of the incident system $\UU\TT_{\widehat{I}}(n,\overline{p})$. From the previous theorem, one can easily deduce the following.

\begin{coro}
    Let $\Gamma=\UU\TT_{\widehat{I}} (n,\overline{p})$ and $F$ the Frobenius morphism of $\overline{\FF_p}$. Then, for all $m>0$ we have $\Gamma^{F^m}\simeq \UU\TT_{\widehat{I}} (n,p^m)$.
\end{coro}

Indeed, with these results, we are able to characterise the coset geometry associated with a pattern group over the algebraic closure of the field.

\begin{theorem}\label{thm:CG_of_Pattern_Fecho}
    Consider a closed set $\widehat{I}\subseteq \Delta_n$, with covering relations $I$, and let $p$ be a prime.
    Let $U:=UT_{\widehat{I}}(n,\overline{\FF_p})$ be the pattern group associated with $\widehat{I}$, and consider the coset incidence system $\UU\TT_{\widehat{I}}(n,\overline{p})=(U,(U_{(i,j)})_{(i,j)\in I})$ where $U_{(i,j)} := \langle e_{k,l}(\alpha) \mid \alpha\in \overline{\FF_p}, (k,l)\in I\setminus\{(i,j)\} \rangle$. Then,
    \begin{enumerate}
        \item $U$ acts flag-transitively on $\UU\TT_{\widehat{I}}(n,\overline{p})$;
        \item $\UU\TT_{\widehat{I}}(n,\overline{p})$ is a residually connected coset geometry;
        \item $\UU\TT_{\widehat{I}}(n,\overline{p})$ is firm;
    \end{enumerate}
\end{theorem}
\begin{proof}

    Let us first prove that $U$ acts flag-transitively on $\UU\TT_{\widehat{I}}(n,\overline{p})$. Consider a flag $\{g_{i,j}U_{(i,j)}\}_{(i,j)\in J}$ of type $J\subseteq I$ with $g_{i,j}\in  UT_{\widehat{I}}(n,\overline{\FF_p})$.
    Since $\overline{\FF_p}=\bigcup_{m\in\mathbb{N}}\FF_{p^m}$, there exists a $q=p^m$ be such that $g_{i,j}\in  UT_{\widehat{I}}(n,\FF_q)$ for all $(i,j)\in J$. Since $g_1U_{(i,j)}(\overline{\FF_p})\cap g_2U_{(s,t)}(\overline{\FF_p})\neq \emptyset$ if and only if $g_1U_{(i,j)}(q)\cap g_2U_{(s,t)}(q)\neq \emptyset$, we have that $\{g_{i,j}U_{(i,j)}(q)\}_{(i,j)\in J}$ is also a flag in $\UU\TT_{\widehat{I}}(n,q)$. 
    By Theorem~\ref{thm:CG_of_Pattern} we have that $UT_{\widehat{I}}(n,\FF_q)$ acts flag-transitively in $\UU\TT_{\widehat{I}}(n,q)$, hence there exists a $h\in UT_{\widehat{I}}(n,\FF_q)$ such that for all $(i,j)\in J$ we have $hg_{i,j}U_{i,j}(q)=U_{i,j}(q)$. This means that we also have $hg_{i,j}U_{i,j}(\overline{\FF_p})=U_{i,j}(\overline{\FF_p})$ for all $(i,j)\in J$. Thus, $U$ acts flag-transitively on $\UU\TT_{\widehat{I}}(n,\overline{p})$.

    For the remaining properties, notice that, for any $(i,j) \in I$, $U_{(i,j)}= \langle U_{(i,j)}(p^m) \mid m \in \mathbb{N}\rangle$. Indeed, for $J\subseteq I$,  is not hard to prove that $U_J = \langle U_J(p^m) \mid m\in \mathbb{N} \rangle$. From the residual-connectedness of each $\UU\TT(n,p^m)$, and Theorem~\ref{thm:CosetRC}(3), we can easily prove the $\UU\TT(n,\overline{p})$ is residually connected.
    To prove firmness, one can follow a similar argument as above.

\end{proof}
\bibliographystyle{ieeetr}
\bibliography{References}

\end{document}